\RequirePackage[2023-11-01]{latexrelease}
\documentclass[hidelinks,onefignum,onetabnum,nohypdvips]{siamart220329}

\usepackage{amsmath,amssymb,mathtools,mathrsfs,bm}
\usepackage{enumitem}
\usepackage{microtype}
\usepackage{cite}

\headers{Feedback Capability for Vector-Valued Systems}{Z. Liu}

\title{A Game-Theoretic Characterization of Feedback Capability for Fully Coupled Vector-Valued Nonparametric Systems\thanks{\textbf{Funding:} This work was supported in part by the National Natural Science Foundation of China under Grant 12401585, the Guangdong Basic and Applied Basic Research Foundation under Grant 2024A1515011542, and the General Program of the Shenzhen Natural Science Foundation under Grant JCYJ20250604181037012.}}

\author{ZHAOBO LIU\thanks{Institute for Advanced Study, Shenzhen University,
Shenzhen, Guangdong 518060, China. Email: \email{liuzhaobo@szu.edu.cn}.}}

\newcommand{\R}{\mathbb{R}}
\newcommand{\Env}{\operatorname{Env}}
\newcommand{\Esc}{\operatorname{Esc}}
\newcommand{\GL}{\operatorname{gL}}
\newcommand{\Rad}{\operatorname{rad}}
\makeatletter
\providecommand{\headerps@out}[1]{}
\@ifundefined{XeTeXrevision}{}{%
  \DeclareRobustCommand{\siam@xetex@url}[1]{\href{#1}{\texttt{\detokenize{#1}}}}%
  \let\url\siam@xetex@url
}
\makeatother

\begin{document}

\maketitle

\begin{abstract}
We study feedback stabilization for the discrete-time system
$x_{t+1}=f(x_t)+u_t+w_{t+1}$ in $\R^d$ with unknown $f$ and arbitrary bounded disturbances.  For scalar plants, the sharp feedback capability threshold under generalized Lipschitz uncertainty is $\frac32+\sqrt2$.  We treat fully coupled vector-valued systems, where scalar order and interval recursion are unavailable and coupling precludes a coordinatewise reduction.  We introduce a response-history escape game whose adversarial objective requires a finite envelope and an unbounded state radius.  Borel determinacy, invariance under a change of finite initial history, and slope monotonicity define a game value $\Gamma_d$.  One causal feedback law stabilizes every plant in the uncertainty class against every bounded disturbance sequence whenever $L<\Gamma_d$.  At every $L\ge\Gamma_d$, each causal feedback law and each initial state admit a globally $\Gamma_d$-Lipschitz plant whose state norms tend to infinity with zero disturbances.  Thus stabilization is possible exactly below the critical value.  We also prove
$\Gamma_d=1+(\log d)/(2d)+O((\log\log d)/d)$, with explicit coefficients in bounds on both sides of the remainder.  Explicit finite-dimensional bounds give $\Gamma_d>1$ in every dimension and $\Gamma_2\ge7/5$. A nearest-neighbor law achieves the planar guarantee.  The values are nonincreasing in dimension, and comparison with the scalar theory yields $\Gamma_1=\frac32+\sqrt2$.
\end{abstract}



\section{Introduction}
Feedback stabilization under nonparametric uncertainty is a basic question about
the capability of feedback.  When the plant is not specified by a parametric
model, one asks how large the uncertainty can be while still allowing a causal
controller to stabilize every plant in the class.  This paper studies the
discrete-time system
\begin{equation}\label{eq:plant}
        x_{t+1}=f(x_t)+u_t+w_{t+1},
        \qquad x_t,u_t,w_t\in\R^d,
\end{equation}
where $f\colon\R^d\to\R^d$ is unknown and the disturbance sequence is arbitrary but
bounded.  Stabilization means boundedness of the closed-loop state sequence.
The objective is to determine the range of nonparametric uncertainty that can be
handled by a single causal feedback law for every bounded disturbance sequence.

For scalar plants, this feedback capability question was resolved by Xie and
Guo \cite{XieGuo2000}.  In the scalar setting,
nonparametric uncertainty is measured by the generalized Lipschitz seminorm,
which records the smallest slope after allowing a finite additive envelope.
The sharp threshold for this slope is $\frac32+\sqrt2$.  Below this value
stabilization is possible, while above it every feedback law has a
counterexample.  Subsequent work developed related stabilization and limitation
results within scalar or structured uncertainty classes,
including critical stability and limitation results for discrete-time adaptive
nonlinear control \cite{Guo1997,XieGuo1999,Ma2008IJOC}, robust variants
\cite{LiXie2006}, polynomial criteria for adaptive stabilizability
\cite{LiXieGuo2006,LiuLi2023}, conditions based on the distribution of regions
with moderate growth \cite{LiuLi2019}, and feedback capability for nonlinearly
parameterized or semiparametric uncertainty classes
\cite{LiGuo2011,HuangGuo2012}.  Related work connects feedback with complexity
\cite{Zames1976,Zames1998} and studies control-oriented identification under
worst-case criteria
\cite{HelmickiJacobsonNett1991,MakilaPartingtonGustafsson1995}.
The search for critical conditions for general uncertain nonlinear control
systems of high dimension or order was explicitly identified as an open problem
in \cite{GuoICM2002}.  Broader accounts of feedback capability and uncertainty
are given in \cite{Guo2002,Guo2020}.

The scalar proof does not extend directly to fully coupled vector-valued
systems.  In the scalar case, it uses the order on
$\R$.  The visited states determine an interval, and the analysis follows the
evolution of interval quantities.  In the vector-valued case, the visited states
are finite configurations in Euclidean space.  Moreover, the uncertainty constraint
links the values generated at different visited states through Euclidean
distances.  A coordinatewise treatment would miss these links across
components.

Our construction removes the known control input from each state transition
and records the remaining plant response together with the visited state.  The
generalized Lipschitz constraint can then be expressed through pairwise
differences between the recorded responses and the corresponding states.  This
description leads to a response-history escape game whose adversarial objective
requires a finite envelope and an unbounded state radius.

We use the escape game to define a value $\Gamma_d$ for each dimension.  For
each slope and each initial history, the escape objective is Borel.  Martin's
Borel determinacy theorem \cite{Martin1975} therefore ensures that one player
has a winning strategy at every slope.  We prove that the same player wins from every finite response history,
and monotonicity in the slope then defines $\Gamma_d$.  The main theorem
establishes the separation between subcritical stabilization and
supercritical impossibility for
\eqref{eq:plant} and identifies the separating value with $\Gamma_d$.
Quantitative bounds show that $\Gamma_d$ is finite and strictly greater than
one.  The proof combines one controller for all slopes below $\Gamma_d$ with a
realization argument
that represents response plays with finite envelope as trajectories of
\eqref{eq:plant}.  The realization step lifts the visited states into a Hilbert
space and applies the Kirszbraun--Valentine extension theorem
\cite{Kirszbraun1934,Valentine1945}.

The contributions are as follows.
\begin{enumerate}[label=(\roman*)]
\item We formulate a response-history escape game in terms of states,
responses, and the pairwise constraints defining a finite envelope.  The
adversary wins when the responses have a finite envelope and the state radius
is unbounded.  We prove that the same player wins from every finite response
history.  Borel determinacy and monotonicity in the slope then define a game
value $\Gamma_d$ for each dimension.
\item We prove that the original fully coupled vector-valued plant problem
admits a feedback capability threshold and identify it with $\Gamma_d$.
A single causal feedback law stabilizes the entire uncertainty class when
$L<\Gamma_d$.  At and above $\Gamma_d$, every causal feedback law and every
initial state admit a globally $\Gamma_d$-Lipschitz plant whose state norms
tend to infinity with zero disturbances.  The proof combines one controller
for all subcritical slopes with a closedness argument for games requiring
separated states and a Lipschitz extension.
\item We give explicit lower bounds above one in every dimension using a
nearest-neighbor law, including $\Gamma_2\ge7/5$.
Two-sided estimates establish
$\Gamma_d=1+(\log d)/(2d)+O((\log\log d)/d)$, with explicit bounds on the
remainder.
\end{enumerate}

Finally, we relate the game value to the scalar theory.  In dimension
one, comparison with the Xie--Guo theorem gives
$\Gamma_1=\frac32+\sqrt2$, and an embedding argument gives
$\Gamma_{d+1}\le\Gamma_d$.  Thus the scalar value is an upper bound in every
dimension.

The rest of the paper is organized as follows.  Section~\ref{sec:model}
introduces the plant model and uncertainty class, formulates the feedback
capability problem, and fixes the notation for finite response histories.
Section~\ref{sec:game} constructs the escape game, defines the game value, and
states the main results.  Sections~\ref{sec:proof-threshold}
and~\ref{sec:proof-gamma-basic} prove the strict threshold theorem, the
nearest-neighbor guarantee, and dimension monotonicity.  The
appendices contain the subcritical
controller for the escape game, the realization lemma in a Hilbert space, and
the Borel determinacy verification.  They also prove that the same player wins
for every finite response history.  Appendices~\ref{app:endpoint}
and~\ref{app:dimension-asymptotic} prove the critical endpoint result and
the asymptotic estimates, respectively.

\section{Problem formulation and response histories}\label{sec:model}
Fix a dimension $d\ge1$, and let $\|\cdot\|$ denote the Euclidean norm.
All logarithms are natural, and $\inf\varnothing=+\infty$.

\subsection{Plant model and stabilization}
Throughout the paper we study the discrete-time plant \eqref{eq:plant}, where
the map $f\colon\R^d\to\R^d$ is unknown and the disturbance sequence is bounded:
\[
        \sup_{t\ge1}\|w_t\|<\infty.
\]
A causal feedback law is a sequence $\Sigma=(\Sigma_t)_{t\ge0}$ of maps
\[
        \Sigma_t\colon(\R^d)^{t+1}\times(\R^d)^t\to\R^d,
        \qquad
        u_t=\Sigma_t(x_0,\ldots,x_t;u_0,\ldots,u_{t-1}).
\]
Following the scalar feedback capability formulation in \cite{XieGuo2000},
we work with arbitrary causal maps of observed histories.  The game strategies
introduced below use the same convention.
Given an initial state $x_0$, a plant $f$, a bounded disturbance sequence
$(w_t)_{t\ge1}$, and a causal feedback law, the plant equation determines the
closed-loop state sequence recursively.  Stabilization means boundedness of
this sequence:
\[
        \sup_{t\ge0}\|x_t\|<\infty.
\]

\subsection{Generalized Lipschitz uncertainty and feedback capability}

\begin{definition}
For $f\colon\R^d\to\R^d$, define its generalized Lipschitz seminorm by
\[
        \|f\|_{\GL,d}
        \triangleq
        \lim_{\alpha\to\infty}
        \sup_{x,y\in\R^d}
        \frac{\|f(x)-f(y)\|}{\|x-y\|+\alpha}.
\]
The limit exists in $[0,\infty]$ because the supremum is nonincreasing as a function of $\alpha$.  This quantity is an extended seminorm.  Bounded functions have value zero.
\end{definition}

For $d=1$, this formula reduces to the scalar generalized Lipschitz seminorm in
\cite{XieGuo2000}.  In higher dimensions we use the same normalization with the
Euclidean norm on $\R^d$.

The next lemma characterizes the seminorm by a Lipschitz bound with a finite additive envelope.

\begin{lemma}\label{lem:glnorm-envelope}
For every $f\colon\R^d\to\R^d$,
\[
        \|f\|_{\GL,d}
        =
        \inf\Bigl\{\ell\ge0\mid
        \exists A<\infty\ \text{such that }
        \|f(x)-f(y)\|\le \ell\|x-y\|+A,\ 
        \forall x,y\in\R^d
        \Bigr\}.
\]
In particular, for every $\ell>\|f\|_{\GL,d}$, there exists $A_f<\infty$ such that
\[
        \|f(x)-f(y)\|\le \ell \|x-y\|+A_f,
        \qquad x,y\in\R^d.
\]
\end{lemma}

\begin{proof}
Set
\[
        \psi(\alpha)\triangleq
        \sup_{x,y\in\R^d}
        \frac{\|f(x)-f(y)\|}{\|x-y\|+\alpha}.
\]
Let $\rho$ denote the infimum on the right-hand side of the displayed formula in the statement.  If $\ell>\|f\|_{\GL,d}$, then $\psi(\alpha_0)\le \ell$ for some $\alpha_0>0$, and therefore
\[
        \|f(x)-f(y)\|
        \le
        \ell (\|x-y\|+\alpha_0)
        =
        \ell \|x-y\|+\ell\alpha_0.
\]
Thus $\rho\le \ell$ for every $\ell>\|f\|_{\GL,d}$.  If $\|f\|_{\GL,d}<\infty$, taking the infimum over all such $\ell$ gives $\rho\le\|f\|_{\GL,d}$.  If $\|f\|_{\GL,d}=\infty$, this inequality is automatic.  Conversely, suppose that
\[
        \|f(x)-f(y)\|\le \ell\|x-y\|+A
        \qquad x,y\in\R^d,
\]
for some $\ell\ge0$ and $A<\infty$.  Taking $x=y$ gives $A\ge0$, and hence, for every $\alpha>0$,
\[
        \psi(\alpha)
        \le
        \sup_{s\ge0}\frac{\ell s+A}{s+\alpha}
        =
        \max\left\{\ell,\frac{A}{\alpha}\right\}.
\]
Taking the limit as $\alpha\to\infty$ gives $\|f\|_{\GL,d}\le \ell$.  Taking the infimum over all admissible $\ell$ gives $\|f\|_{\GL,d}\le\rho$.  The final assertion is the first implication with $A_f=\ell\alpha_0$.
\end{proof}

For $L\ge0$, let
\[
        \mathcal F_d(L)\triangleq
        \{f\colon\R^d\to\R^d\mid\|f\|_{\GL,d}\le L\}.
\]
For $f\in\mathcal F_d(L)$, Lemma~\ref{lem:glnorm-envelope} implies that $f$
maps bounded sets to bounded sets.  Hence, under a bounded disturbance,
boundedness of the state sequence also implies boundedness of the control
sequence through $u_t=x_{t+1}-f(x_t)-w_{t+1}$.  In the scalar case, this is the
stability criterion used in \cite{XieGuo2000}.

We say that $\mathcal F_d(L)$ is stabilizable under bounded disturbances if
there exists a causal feedback law such that the corresponding
closed-loop state sequence is bounded for every initial state, every
$f\in\mathcal F_d(L)$, and every bounded disturbance sequence.

The feedback capability problem is to determine the values of $L$ in the
stabilizable and impossibility regimes.  In the latter
regime, for every causal feedback law and every prescribed initial state, there
exist $f\in\mathcal F_d(L)$ and a bounded disturbance sequence for which the
resulting closed-loop state sequence is unbounded.

\subsection{Response histories and finite envelopes}

Along a trajectory of \eqref{eq:plant}, define the response at time $t$ by
\begin{equation}\label{eq:response}
        r_t\triangleq x_{t+1}-u_t=f(x_t)+w_{t+1}.
\end{equation}
Since the applied control is known, this response is known once the successor
state $x_{t+1}$ has been observed.  We use the states and responses to define
the finite response histories used below.

\begin{definition}
A finite response history of length $n\ge0$ has the form
\[
        H=(x_0,\ldots,x_n;r_0,\ldots,r_{n-1}),
\]
where $x_0,\ldots,x_n\in\R^d$ are states and, for $0\le i<n$,
$r_i\in\R^d$ is the response associated with $x_i$.  When $n=0$ the response
list is empty.  Its state radius is
\[
        \Rad(H)\triangleq\max_{0\le i\le n}\|x_i-x_0\|.
\]
If a control value $u\in\R^d$ is applied at $x_n$ and the response is
$r\in\R^d$, the successor history is
\[
        H^+(u,r)\triangleq(x_0,\ldots,x_n,u+r;r_0,\ldots,r_{n-1},r).
\]
\end{definition}

\begin{definition}
For a finite response history $H$ and a slope $\ell\ge0$, define
\begin{equation}\label{eq:Env}
        \Env_\ell(H)\triangleq
        \begin{cases}
        0, & n=0,\\[1mm]
        \displaystyle
        \max_{0\le i,j\le n-1}
        \bigl(\|r_i-r_j\|-\ell\|x_i-x_j\|\bigr)_+, & n\ge 1
        \end{cases},
\end{equation}
where $q_+\triangleq\max\{q,0\}$ for $q\in\R$.
The last state $x_n$ has no recorded response in $H$, so the maximum in
\eqref{eq:Env} ranges over $0,\ldots,n-1$.
Thus $\Env_\ell(H)$ is the smallest $B\ge0$ such that
\[
        \|r_i-r_j\|\le \ell\|x_i-x_j\|+B,
        \qquad 0\le i,j\le n-1,
\]
with the convention that this condition is vacuous when $n=0$.
\end{definition}

\section{The response-history escape game and main results}\label{sec:game}
This section defines the response-history escape game and its value $\Gamma_d$,
then states the main results.

\begin{definition}\label{def:play}
Starting from an initial history $H^0$, a play is a sequence
\[
        \varpi=(H^k,u_k,\widehat r_k,H^{k+1})_{k\ge 0},
        \qquad
        H^{k+1}=(H^k)^+(u_k,\widehat r_k).
\]
At stage $k$, the controller observes $H^k$ and chooses $u_k\in\R^d$.
The adversary then observes $(H^k,u_k)$ and chooses $\widehat r_k\in\R^d$.
If $H^0$ has length $n$, then $H^k$ has length $n+k$, and
$\widehat r_k$ is the response $r_{n+k}$ in the full history.
When $n=0$, we omit the hat.
A controller strategy is a causal map $\mu$ from finite response histories
to $\R^d$, so $u_k=\mu(H^k)$. An adversary strategy is a causal map
$\sigma$ from pairs $(H,u)$ to $\R^d$, so
$\widehat r_k=\sigma(H^k,u_k)$. Thus a pair of strategies and an initial
history determine a unique play.
\end{definition}

Let $\mathscr P(H^0)$ denote the set of all plays starting from $H^0$.  For a slope $\ell\ge0$ and a play $\varpi$, define
\[
        \Env_\ell^\infty(\varpi)\triangleq\sup_{k\ge 0}\Env_\ell(H^k).
\]

\begin{definition}\label{def:escape-objective}
Fix $\ell\ge0$.  The adversary's escape objective is
\[
        \Esc_\ell\triangleq
        \Bigl\{\varpi\mid
        \Env_\ell^\infty(\varpi)<\infty
        \text{ and }
        \sup_k \Rad(H^k)=\infty
        \Bigr\}.
\]
When the initial history $H^0$ is fixed, write
$\Esc_{\ell,H^0}\triangleq\Esc_\ell\cap \mathscr P(H^0)$.
\end{definition}

\begin{definition}\label{def:winning}
Fix an initial history $H^0$ and a target set $\Omega\subseteq\mathscr P(H^0)$.
\begin{enumerate}[label=(\roman*)]
\item The adversary wins from $H^0$ for the target $\Omega$ if there exists an adversary strategy $\sigma$ such that, for every controller strategy $\mu$, the resulting play belongs to $\Omega$.
\item The controller wins from $H^0$ against $\Omega$ if there exists a controller strategy $\mu$ such that, for every adversary strategy $\sigma$, the resulting play does not belong to $\Omega$.
\end{enumerate}
\end{definition}

\begin{proposition}\label{prop:escape-determinacy}
For every finite response history $H^0$ and every slope $\ell\ge0$, the escape
game from $H^0$ at slope $\ell$ is determined.  Hence exactly one player has a
winning strategy for $\Esc_{\ell,H^0}$.
\end{proposition}

The proof is given in Appendix~\ref{app:escape-game-application}.

At each fixed slope, determinacy classifies every finite response history
according to which player has a winning strategy.  Let $\mathcal U_\ell$ be
the set of finite response histories from which the adversary wins:
\[
        \mathcal U_\ell\triangleq
        \{H\mid \text{the adversary wins from }H\text{ for }\Esc_{\ell,H}\}.
\]
If $\ell_1\le\ell_2$, then, for every finite response history $H$,
\[
        \Env_{\ell_2}(H)\le \Env_{\ell_1}(H),
        \qquad
        \Esc_{\ell_1,H}\subseteq \Esc_{\ell_2,H}.
\]
An adversary winning at slope $\ell_1$ also wins at slope $\ell_2$.  Hence
$\mathcal U_{\ell_1}\subseteq\mathcal U_{\ell_2}$, and the set of slopes at
which the controller wins from a fixed initial history is downward closed.

\begin{proposition}\label{prop:finite-history-invariance}
For every $\ell\ge0$, either $\mathcal U_\ell=\varnothing$ or
$\mathcal U_\ell$ is the set of all finite response histories.
\end{proposition}

The proof is given in Appendix~\ref{app:finite-initial-histories}.
It translates the future play between two initial histories.  The finite
prefixes contribute only a finite additive term to the envelope.

By slope monotonicity, the slopes for which $\mathcal U_\ell$ is empty form a
downward closed set.  Define
\begin{equation}\label{eq:Gamma}
        \Gamma_d\triangleq
        \sup\bigl(\{0\}\cup\{\ell\ge0\mid\mathcal U_\ell=\varnothing\}\bigr).
\end{equation}
The supremum in \eqref{eq:Gamma} is allowed to be $+\infty$.

The next theorem establishes the strict separation between subcritical
stabilization and supercritical impossibility and identifies the separating
value with $\Gamma_d$.

\begin{theorem}\label{thm:nf-critical}
For the fully coupled vector-valued uncertain system \eqref{eq:plant}, the
following hold.
\begin{enumerate}[label=(\alph*)]
\item \emph{Subcritical stabilization.} There exists a single causal feedback
law $\Sigma^*$ such that, for every $0\le L<\Gamma_d$, every initial state
$x_0\in\R^d$, every plant $f\in\mathcal F_d(L)$, and
every bounded disturbance sequence $(w_t)_{t\ge1}$, the closed-loop state
sequence is bounded.
\item \emph{Supercritical impossibility.} If $L>\Gamma_d$, then for every
causal feedback law $\Sigma$ and every initial state $x_0\in\R^d$, there exist
$f\colon\R^d\to\R^d$ with $\|f\|_{\GL,d}<L$ and a bounded disturbance sequence
$(w_t)_{t\ge1}$ such that the closed-loop state sequence from $x_0$ under
$\Sigma$ is unbounded.
\end{enumerate}
\end{theorem}

The proof is given in Section~\ref{sec:proof-threshold}.

Theorem~\ref{thm:critical-endpoint} below treats equality and strengthens
part~(b).

For $d\ge3$, define
\[
 p_d\triangleq\frac{1}{2(d-1)}
       \left(\frac{d-1}{4d}\right)^{(d-1)/2},
 \qquad
 \Lambda_d\triangleq\min\!\left\{
   \sqrt{\frac{d}{d-1}},   \exp\!\left[p_d\left(1+\log\frac1{2p_d}\right)\right]
 \right\}.
\]
Set $\Lambda_1\triangleq2$ and $\Lambda_2\triangleq7/5$.

\begin{theorem}\label{thm:gamma-basic}
For every $d\ge1$, $\Gamma_d\ge\Lambda_d>1$ and
$\Gamma_{d+1}\le\Gamma_d$.
\end{theorem}

Section~\ref{sec:proof-gamma-basic} proves the bound $\Lambda_d$ and
dimension monotonicity.

\begin{corollary}\label{cor:explicit-nn-controller}
For every $d\ge1$, there is an explicit nearest-neighbor causal feedback law
such that, for every $0\le L<\Lambda_d$, every initial state $x_0\in\R^d$, every
plant $f\in\mathcal F_d(L)$, and every bounded disturbance sequence, the
closed-loop state sequence is bounded.  The same law applies to all
$L<\Lambda_d$ and uses neither $L$, $f$, nor a bound on the disturbance.
\end{corollary}

The proof is given in Section~\ref{sec:proof-gamma-basic}.  Since $\Lambda_d>1$,
this explicit law covers slopes strictly larger than one in every finite
dimension.

\begin{corollary}\label{cor:universal-interval}
For every $d\ge1$,
\begin{equation}\label{eq:universal-interval}
        \Lambda_d\le \Gamma_d\le \Gamma_1=\frac32+\sqrt2.
\end{equation}
In particular,
\[
        \frac75
        \le \Gamma_2
        \le \frac32+\sqrt2.
\]
\end{corollary}

\begin{proof}
The scalar result \cite[Theorem~2.1]{XieGuo2000} treats the same uncertainty
class and stabilization problem under bounded disturbances.  The causal feedback
classes are equivalent because past controls generated by a fixed law can be
recovered recursively from the state history.  It gives the sharp threshold
$\ell_*=\frac32+\sqrt2$.  If $\Gamma_1<\ell_*$, choose
$\Gamma_1<L<\ell_*$.  Theorem~\ref{thm:nf-critical}(b) contradicts scalar
stabilization at slope $L$.  If $\ell_*<\Gamma_1$, choose
$\ell_*<L<\Gamma_1$.  Theorem~\ref{thm:nf-critical}(a) then contradicts scalar
impossibility at slope $L$.  Hence $\Gamma_1=\ell_*$.

Together with the lower bound and dimension monotonicity in
Theorem~\ref{thm:gamma-basic}, this proves \eqref{eq:universal-interval}.

\end{proof}

The next result completes the threshold classification at equality.

\begin{theorem}\label{thm:critical-endpoint}
For every finite dimension $d$, the adversary wins the escape game at slope
$\Gamma_d$ from every finite response history.  Moreover, for every causal
feedback law $\Sigma$ and every initial state $a\in\R^d$, there is a map
$f\colon\R^d\to\R^d$ satisfying
\[
 \|f(x)-f(y)\|\le\Gamma_d\|x-y\|\qquad x,y\in\R^d
\]
such that the closed-loop trajectory from $a$ with $w_t=0$ for all $t\ge1$
satisfies $\|x_t\|\to\infty$.
Consequently, the class $\mathcal F_d(L)$ is stabilizable against every
bounded disturbance sequence by one causal feedback law if and only if
$L<\Gamma_d$.
\end{theorem}

The proof is given in Appendix~\ref{app:endpoint}.  The ordinary Lipschitz
bound implies $\|f\|_{\GL,d}\le\Gamma_d$, so the counterexample belongs to
$\mathcal F_d(L)$ for every $L\ge\Gamma_d$.

\begin{theorem}\label{thm:dimension-asymptotic}
For every integer $d\ge2$,
\begin{equation}\label{eq:dimension-remainder}
 -3\log\log d-32
 \le d(\Gamma_d-1)-\frac12\log d
 \le \frac32\log\log d+19 .
\end{equation}
In particular, as $d\to\infty$,
\begin{equation}\label{eq:dimension-leading}
 \Gamma_d=1+\frac{\log d}{2d}
             +O\!\left(\frac{\log\log d}{d}\right),
 \qquad
 \lim_{d\to\infty}\frac{d(\Gamma_d-1)}{\log d}=\frac12.
\end{equation}
\end{theorem}

Appendix~\ref{app:dimension-asymptotic} proves the theorem.
For example, its bound \eqref{eq:finite-window-upper} gives
$\Gamma_{100}<1.141$ and $\Gamma_{1000}<1.01975$.

\section{Proof of Theorem~\ref{thm:nf-critical}}\label{sec:proof-threshold}
The proof has two directions.  In the subcritical direction, the escape game
controller from Lemma~\ref{lem:no-slope-controller} is implemented as a plant
feedback law by using \eqref{eq:response} to reconstruct the responses.  The
generalized Lipschitz bound on the plant and boundedness of the disturbance imply that the
induced response play has finite envelope at a subcritical slope, and the
controller's winning property then gives bounded states.  In the supercritical
direction, an adversary winning strategy in the escape game produces an
escaping response play with finite envelope.  Lemma~\ref{lem:nf-realization},
proved through a realization in a Hilbert space and the Kirszbraun--Valentine
extension theorem, realizes that play as the trajectory of
an actual plant with generalized Lipschitz seminorm below the prescribed
supercritical slope and with bounded disturbances.

For part~(a), let $\mu^*$ be the escape game strategy from
Lemma~\ref{lem:no-slope-controller}.  At time $t$, reconstruct the observed
responses through \eqref{eq:response}, form
$H^t=(x_0,\ldots,x_t;r_0,\ldots,r_{t-1})$, and apply
$u_t=\mu^*(H^t)$.  Denote the resulting causal feedback law for
\eqref{eq:plant} by $\Sigma^*$.

Fix $0\le L<\Gamma_d$, an initial state $x_0\in\R^d$, a plant
$f\in\mathcal F_d(L)$, and a bounded disturbance sequence $(w_t)_{t\ge1}$.
Choose a finite slope
\[
        \|f\|_{\GL,d}\le L<\alpha<\Gamma_d.
\]
By Lemma~\ref{lem:glnorm-envelope}, there exists $A_f<\infty$ such that
\[
        \|f(x)-f(y)\|\le \alpha\|x-y\|+A_f,
        \qquad x,y\in\R^d.
\]
Write $W=\sup_{t\ge1}\|w_t\|<\infty$.  By \eqref{eq:response}, the closed-loop
state sequence induces a response play $\varpi$ controlled by $\mu^*$, and for all
$i,j\ge0$,
\[
        \|r_i-r_j\|
        \le
        \|f(x_i)-f(x_j)\|+\|w_{i+1}\|+\|w_{j+1}\|
        \le
        \alpha\|x_i-x_j\|+A_f+2W.
\]
Consequently,
\[
        \Env_\alpha^\infty(\varpi)\le A_f+2W<\infty.
\]
Lemma~\ref{lem:no-slope-controller} gives boundedness of the state sequence.
This proves part~(a).

For part~(b), the assertion is vacuous if $\Gamma_d=\infty$.  Otherwise fix $L>\Gamma_d$ and choose
\[
        \Gamma_d<\ell<K<L.
\]
Equation~\eqref{eq:Gamma} gives $\mathcal U_\ell\ne\varnothing$.  Hence
Proposition~\ref{prop:finite-history-invariance} implies that every finite
response history belongs to $\mathcal U_\ell$.

Fix a causal feedback law $\Sigma=(\Sigma_t)_{t\ge0}$ for the plant and an initial
state $x_0$.  We view $\Sigma$ as a controller on response histories by recovering
past controls from a response history.  For
$H=(z_0,\ldots,z_n;r_0,\ldots,r_{n-1})$, set
$u_i=z_{i+1}-r_i$ for $0\le i<n$ and define
\[
        \mu(H)=\Sigma_n(z_0,\ldots,z_n;u_0,\ldots,u_{n-1}).
\]
For $n=0$ this means $\mu((z_0))=\Sigma_0(z_0)$.  A response play compatible
with $\mu$ therefore realizes exactly the controls prescribed by $\Sigma$ along
the same state sequence.

Let $H^0=(x_0)$.  Since $H^0\in\mathcal U_\ell$, an adversary strategy produces
a play $\varpi\in\Esc_{\ell,H^0}$ against $\mu$.  Thus $\varpi$ is compatible with
$\mu$, has finite envelope at slope $\ell$, and has an unbounded state radius.

Apply Lemma~\ref{lem:nf-realization} to this play.  Since it has finite envelope
at slope $\ell$, the lemma realizes it using a $K$-Lipschitz map
$f\colon\R^d\to\R^d$ and a bounded disturbance sequence $(w_t)_{t\ge1}$.
Because $\mu$ was induced from $\Sigma$ by recovering past controls, the same
unbounded state sequence is realized by \eqref{eq:plant} under the original
feedback law $\Sigma$.  Moreover, $\|f\|_{\GL,d}\le K<L$, so
$f\in\mathcal F_d(L)$.  Thus the constructed plant and bounded disturbance
provide the required counterexample to $\Sigma$.
This proves part~(b), and hence Theorem~\ref{thm:nf-critical}.

\section{Finite-dimensional bounds and monotonicity}\label{sec:proof-gamma-basic}
We first prove the guarantee for nearest-neighbor cancellation.

\begin{lemma}\label{lem:nearest-neighbor-packing}
Fix $d\ge1$ and $0\le\ell<\Lambda_d$.
Start from an arbitrary one-step response history $H=(x_0,x_1;r_0)$.  At every
time $t\ge1$, set
\begin{equation}\label{eq:nearest-neighbor-index}
        \delta_t\triangleq\min_{0\le i<t}\|x_t-x_i\|,
        \qquad
        i_t\triangleq
        \min\{i\mid 0\le i<t,\ \|x_t-x_i\|=\delta_t\},
\end{equation}
and apply the control
\begin{equation}\label{eq:nearest-neighbor-control}
        u_t=x_0-r_{i_t}.
\end{equation}
Then every play generated by this controller with
$\Env_\ell^\infty(\varpi)<\infty$ has a bounded state sequence.
\end{lemma}

\begin{proof}
Let a generated play satisfy $\Env_\ell^\infty(\varpi)\le A<\infty$.  Then, for all
response indices $i,j\ge0$,
\begin{equation}\label{eq:pairwise-response-NN}
        \|r_i-r_j\|\le \ell\|x_i-x_j\|+A.
\end{equation}
By \eqref{eq:nearest-neighbor-control}, the successor relation gives
$x_{t+1}-x_0=r_t-r_{i_t}$.  Hence \eqref{eq:pairwise-response-NN} implies
\begin{equation}\label{eq:novelty-recursion}
        \|x_{t+1}-x_0\|\le \ell\delta_t+A,
        \qquad t\ge1.
\end{equation}

If $(\delta_t)$ is bounded, then \eqref{eq:novelty-recursion} makes the state
sequence bounded.  Assume, for contradiction, that $(\delta_t)$ is unbounded.
Enumerate the positive strict record times of
$(\delta_t)$ as $n_1<n_2<\cdots$.  Explicitly,
these are the indices satisfying
\[
        \delta_{n_k}>
        \max\bigl(\{0\}\cup\{\delta_t\mid 1\le t<n_k\}\bigr).
\]
Set
\[
        m_k\triangleq\delta_{n_k},
        \qquad
        y_k\triangleq x_{n_k}-x_0.
\]
Then $0<m_1<m_2<\cdots$, $m_k\to\infty$, and, for $k\ge2$, every
$t<n_k$ satisfies $\delta_t\le m_{k-1}$.

Choose $\max\{1,\ell\}<q<\Lambda_d$.
Since $m_k\to\infty$, there is $k_0\ge2$ such that
$A\le(q-\ell)m_{k-1}$ for every $k\ge k_0$.
For such $k$, $x_0$ is one of the states preceding $x_{n_k}$, and therefore
$m_k\le\|y_k\|$.  Applying
\eqref{eq:novelty-recursion} at time $n_k-1$ gives
\begin{equation}\label{eq:record-ratio-upper}
        m_k
        \le \|y_k\|
        \le \ell\delta_{n_k-1}+A
        \le \ell m_{k-1}+A
        \le q m_{k-1}.
\end{equation}
For $d=1$, $q<2$. Records $i<j$ of the same sign, with $i,j\ge k_0$,
satisfy $|y_j|\ge|y_i|+m_j$ by separation, since $|y_i|<2m_j$.
Thus consecutive records have opposite signs by
\eqref{eq:record-ratio-upper}, and
$m_j+m_{j-2}\le q m_{j-1}$ for $j\ge k_0+2$.
Writing $\Delta_j\triangleq m_j-m_{j-1}>0$ for $j\ge k_0+1$ gives
\[
 \Delta_j\le\Delta_{j-1}-(2-q)m_{j-1}\qquad j\ge k_0+2.
\]
Thus the positive increments decrease by at least $(2-q)m_{k_0}>0$
at each step, a contradiction.
Hence assume $d\ge2$.

Choose an axis $\omega$ uniformly on the unit sphere. For a fixed
$0<\alpha<\pi/4$, select each record whose direction makes angle less
than $\alpha$ with $\omega$. Write $b_j=\|y_j\|/m_j\in[1,q]$ and choose
$0<c\le\cos(2\alpha)$. Suppose
\[
 2q^2(1-c)<1,\qquad q<2c.
\]
Two selected sites, normalized by their own record scales, are at
distance less than one. Indeed, their squared distance is bounded by
$a^2+b^2-2abc$ for $a,b\in[1,q]$. This expression is convex in each
variable. Its two equal-radius corners are bounded by $2q^2(1-c)$, and
the mixed corners are bounded by $1+q^2-2qc<1$.

Let $i<j$ be consecutive selected records, and put $t=\|y_i\|/\|y_j\|$
and $b=b_j$. If the cosine of their angle is at least $c'\ge c$, the
separation $\|y_i-y_j\|\ge m_j$ gives
$t^2-2c't+1-b^{-2}\ge0$. At the unscaled ratio $b_i/b$ the same
quadratic is strictly negative, by the preceding distance estimate.
Since $m_i<m_j$, we have $t<b_i/b$, so
\begin{equation}\label{eq:cap-root}
 t\le \rho(b,c'),\qquad \rho(b,c')\triangleq c'-\sqrt{(c')^2-1+b^{-2}}.
\end{equation}
In particular, $b=1$ cannot occur after the first selected record.

For each axis, charge only selected records, with each charge bounded
above by a fixed constant and at most $-\log t$ after the first
selection. The sum through record $n$, starting at $k_0$, is then at
most $\log(m_n/m_{k_0})+O(1)$, uniformly in the axis. This follows by
telescoping the physical radii and using $m_j\le\|y_j\|\le q m_j$.
On the other hand, the record growth gives
\[
 \sum_{j=k_0}^n \log(q/b_j)\ge\log(m_n/m_{k_0-1}).
\]
Thus it suffices to give charges whose expectation at every record
exceeds $\log(q/b_j)$ by a fixed positive amount.

For $d\ge3$, take $\sin^2\alpha=(d-1)/(4d)$ and
$c=\cos(2\alpha)=(d+1)/(2d)$. Since $q<\Lambda_d\le\sqrt{d/(d-1)}<2c$, the two
corner conditions hold. At $b=\sqrt{d/(d-1)}$, the smaller root
$\rho(b,c)$ is $1/d$. Thus $h=\rho(b,c)\le1/d$ throughout the interval
and
\[
 \log b=-\frac12\log(1-2ch+h^2)\ge ch-h^2/2\ge h/2.
\]
The probability of selection is at least $p_d$. In the polar angle
formula, the denominator $\int_0^\pi\sin^{d-2}v\,dv$ is at most two,
and the numerator $\int_0^\alpha\sin^{d-2}v\,dv$ is at least
$\sin^{d-1}\alpha/(d-1)$, by integrating the derivative of
$\sin^{d-1}v$. Choose a truncation constant $M>0$ and charge
$\min\{M,-\log(2\log b)\}$ inside the cap, with
the value $M$ at $b=1$. The charge is nonnegative and is at most
$-\log t$ after the first selection by \eqref{eq:cap-root}.
For $x=\log b>0$,
\[
 x-p_d\log(2x)\ge p_d\left(1+\log\frac1{2p_d}\right)>\log q,
\]
by minimization at $x=p_d$. Where the charge is truncated, choosing
$p_d M>\log q$ also gives a uniform positive margin.

For $d=2$, take $\alpha=3\pi/26$ and $c=\sqrt5/3<\cos(2\alpha)$.
The corner conditions hold even for $q=7/5$. For a selected record at unsigned
angle $v\in[0,\alpha]$ from the axis, take $c'=\cos(\alpha+v)$ in
\eqref{eq:cap-root}. The quadratic in that equation evaluated at
$(b-1)/c'$ is nonpositive, since its value multiplied by $b^2$ is
\[
 (b-1)^2\left(\frac{b^2}{(c')^2}-2b-1\right)\le0.
\]
The expression in parentheses increases with $b\ge1$ and decreases
with $c'$, and is negative at $b=7/5$ and $c'=\sqrt5/3$.
Choose a truncation constant $M\ge1$. We can therefore charge
$\log c'+\min\{M,-\log(b-1)\}$, with the same
convention at $b=1$. Conditional on selection, $v$ is uniform on
$[0,\alpha]$. By concavity of $\log\cos$ on $[\alpha,2\alpha]$,
the mean of $\log c'$ is at least
$\frac12(\log\cos\alpha+\log c)>\log(5/6)$.
The selection probability is exactly $p\triangleq3/26$.
The bound $\cos(2\alpha)>c$ follows from $\pi<22/7$ and
the Taylor lower bound for cosine through degree six.
The logarithmic mean comparison follows from
$c^2(1+c)/2>(5/6)^4$.

The planar charge is nonnegative. When it is not
truncated, the expected charge plus $\log b$ is at least
\[
 \log b+p\log\frac{5/6}{b-1}
 \ \ge\ p\log\frac{5/6}{p}-(1-p)\log(1-p)
 \ >\ \log(7/5).
\]
The minimum occurs at $b=1/(1-p)$. The strict comparison is the
rational inequality
$(65/9)^3(26/23)^{23}>(7/5)^{26}$.
At truncated records, enlarge $M$ so that
$p(M+\log(5/6))>\log(7/5)$.
This supplies a uniform positive margin in the plane as well.

In either case the expected total charge would eventually exceed the
upper bound from telescoping. Hence the novelty records cannot be
unbounded. The one-step estimate
\eqref{eq:novelty-recursion} then bounds the state sequence.
\end{proof}

For any $d$, the preceding controller prevents escape with a finite
envelope at every $\ell<\Lambda_d$. Proposition~\ref{prop:finite-history-invariance}
and \eqref{eq:Gamma} give $\Gamma_d\ge\Lambda_d$.

It remains to prove monotonicity.  If $\Gamma_d=\infty$, then
$\Gamma_{d+1}\le\Gamma_d$ is immediate.  Assume $\Gamma_d<\infty$.  Let
$I\colon\R^d\to\R^{d+1}$ be $Iz=(z,0)$, and let
$P\colon\R^{d+1}\to\R^d$ be the projection onto the first $d$ coordinates.
Let $H_*^{(j)}=(0)$ be the history of length zero in dimension $j$, and write
$\mathcal U_\ell^{(j)}$ for the corresponding winning set.  Fix
$\ell>\Gamma_d$.  Equation~\eqref{eq:Gamma} gives
$\mathcal U_\ell^{(d)}\ne\varnothing$, and
Proposition~\ref{prop:finite-history-invariance} then yields
$H_*^{(d)}\in\mathcal U_\ell^{(d)}$.

Let $\bar\sigma$ be an adversary winning strategy from $H_*^{(d)}$ at slope
$\ell$.  For a finite response history $H$ in dimension $d+1$, let $PH$ denote
the history obtained by applying $P$ to all its states and responses.  Define a
$(d+1)$-dimensional adversary strategy from $H_*^{(d+1)}$ by
\[
        \sigma(H,u)=I\,\bar\sigma(PH,Pu).
\]

Fix any $(d+1)$-dimensional controller strategy $\mu$, and let
$\varpi=(H^t,u_t,r_t,H^{t+1})_{t\ge0}$ be the play generated by $\mu$ and
$\sigma$.  Write $\bar x_i=Px_i$, $\bar r_i=Pr_i$, and $\bar H^t=PH^t$, and
let
$\bar\varpi=(\bar H^t,Pu_t,Pr_t,\bar H^{t+1})_{t\ge0}$.  To compare $\bar\varpi$
with the $d$-dimensional game, define a $d$-dimensional controller strategy
$\bar\mu$ on the histories appearing in $\bar\varpi$ by
\[
        \bar\mu(\bar H^t)=P\,\mu(H^t),\qquad t\ge0.
\]
After $\mu$ is fixed, the play $\varpi$ is uniquely determined.  The histories
$\bar H^t$ have distinct lengths, so these prescriptions are consistent.  Set
$\bar\mu(\bar H)=0$ on all other $d$-dimensional finite response histories.
The completed map $\bar\mu$ is a $d$-dimensional controller strategy, and
$\bar\sigma$ wins against every such strategy.

The play generated by $\bar\mu$ and $\bar\sigma$ from $H_*^{(d)}$ is exactly
$\bar\varpi$, because at each time $t$,
\[
\begin{aligned}
        \bar\mu(\bar H^t)&=P\mu(H^t)=Pu_t,\\
        P\sigma(H^t,u_t)
        &=\bar\sigma(PH^t,Pu_t)
          =\bar\sigma(\bar H^t,\bar\mu(\bar H^t)).
\end{aligned}
\]
Since $\bar\sigma$ is winning, $\bar\varpi$ has finite envelope at slope $\ell$ and an unbounded state radius.

For all $i,j$,
\[
        \|r_i-r_j\|_{\R^{d+1}}
        =
        \|\bar r_i-\bar r_j\|_{\R^d},
        \qquad
        \|x_i-x_j\|_{\R^{d+1}}
        \ge
        \|\bar x_i-\bar x_j\|_{\R^d}.
\]
Therefore, for every finite time $k$,
\[
        \Env_\ell(H^k)\le \Env_\ell(\bar H^k).
\]
Since $\bar\varpi$ has finite envelope at slope $\ell$, the $(d+1)$-dimensional
play $\varpi$ also has finite envelope at slope $\ell$.  Its state radius is unbounded
because
$\|x_i-x_0\|_{\R^{d+1}}\ge \|\bar x_i-\bar x_0\|_{\R^d}$ for every $i$.
Since the argument applies to every
$(d+1)$-dimensional controller strategy $\mu$, the strategy $\sigma$ forces
$\Esc_{\ell,H_*^{(d+1)}}$ from $H_*^{(d+1)}$.  Thus
$H_*^{(d+1)}\in\mathcal U_\ell^{(d+1)}$.  Slope monotonicity and
\eqref{eq:Gamma} give $\Gamma_{d+1}\le\ell$.

Since this holds for every $\ell>\Gamma_d$, we conclude $\Gamma_{d+1}\le\Gamma_d$.
This proves the dimension monotonicity in Theorem~\ref{thm:gamma-basic}.

\begin{proof}[Proof of Corollary~\ref{cor:explicit-nn-controller}]
Set $u_0=0$ and use
\eqref{eq:nearest-neighbor-index}--\eqref{eq:nearest-neighbor-control} for
$t\ge1$.  This rule is causal because all past responses are determined by the
observed history.  Fix $0\le L<\Lambda_d$,
$f\in\mathcal F_d(L)$, and a disturbance sequence satisfying
$W=\sup_{t\ge1}\|w_t\|<\infty$.  Choose $L<\ell<\Lambda_d$.
Lemma~\ref{lem:glnorm-envelope} and \eqref{eq:response} show that the induced
response play $\varpi$ satisfies
$\Env_\ell^\infty(\varpi)\le A_f+2W$ for some $A_f<\infty$.
Lemma~\ref{lem:nearest-neighbor-packing} therefore proves boundedness of the
closed-loop state sequence.  The feedback rule contains none of $L$, $f$, or
$W$, although the resulting state bound may depend on the initial state, the
plant, and the disturbance bound.
\end{proof}

\section{Conclusion}\label{sec:conclusion}
This paper characterizes the feedback capability of fully coupled
vector-valued systems through a response-history escape game.  One causal
feedback law stabilizes the entire class $\mathcal F_d(L)$ against every
bounded disturbance sequence exactly when $L<\Gamma_d$.  At equality and
above, every causal feedback law and every initial state admit an ordinary
$\Gamma_d$-Lipschitz plant whose state norms tend to infinity with zero
disturbances.  The values are nonincreasing in dimension, satisfy
$\Gamma_1=\frac32+\sqrt2$, and obey
$\Gamma_d=1+(\log d)/(2d)+O((\log\log d)/d)$, with the bounds
\eqref{eq:dimension-remainder} on the remainder.  The explicit
nearest-neighbor law gives $\Gamma_d\ge\Lambda_d>1$ in every finite
dimension, including $\Gamma_2\ge7/5$.  Further questions concern the
value in a fixed dimension $d\ge2$ and a sharper second term in the
dimension asymptotic.

\appendix

\section{One controller for all subcritical slopes}\label{sec:slopefree}
For each fixed slope below $\Gamma_d$, the definition of $\Gamma_d$ and
determinacy give a controller strategy that prevents escape from the initial
history $(0)$.  This appendix combines these strategies into one controller
that is chosen before the slope and the envelope bound of the generated play
are known.  Translation allows the same construction to restart from any
current state.  Section~\ref{sec:proof-threshold} implements the resulting
controller in the plant model.

\begin{lemma}\label{lem:no-slope-controller}
There exists a single causal controller strategy $\mu^*$ with the following
property.  For every initial history $H^0=(x_0)$ with $x_0\in\R^d$ and every
play $\varpi$ starting from $H^0$ whose controller moves are prescribed by
$\mu^*$, if
\[
        \Env_\alpha^\infty(\varpi)<\infty
        \qquad
        \text{for some }0\le\alpha<\Gamma_d,
\]
then the state sequence in $\varpi$ is bounded.
\end{lemma}

\begin{proof}
If $\Gamma_d=0$, the hypothesis $0\le\alpha<\Gamma_d$ is never satisfied, so
any causal controller strategy has the stated property.  Assume from now on
that $\Gamma_d>0$.

Choose an increasing sequence of design slopes below $\Gamma_d$:
\[
        0\le \ell_1<\ell_2<\cdots<\Gamma_d,
        \qquad
        \lim_{m\to\infty}\ell_m=\Gamma_d
        \quad\text{if }\Gamma_d<\infty,
\]
and, if $\Gamma_d=\infty$, choose instead $0\le \ell_1<\ell_2<\cdots$ with
$\ell_m\to\infty$.  For each $m$, equation~\eqref{eq:Gamma} and slope
monotonicity give $\mathcal U_{\ell_m}=\varnothing$.  With
$H_*=(0)$, Proposition~\ref{prop:escape-determinacy} therefore provides a
controller strategy $\mu_m$ that wins from $H_*$ against
$\Esc_{\ell_m,H_*}$.  Set $B_m=m$.  The sequences $(\ell_m)$ and $(B_m)$ and
the strategies $(\mu_m)$ are fixed before any play begins.

\medskip
\noindent\textit{Step 1. Stage definition.}
Stage $1$ starts at time $\tau_1=0$.  Suppose stage $m$ starts at time $\tau_m$.
Set $p_m=x_{\tau_m}$.  For each $k\ge0$ that occurs before the next stage
starts, define the stage subhistory and its translate by
\[
\begin{aligned}
        G_m^k&=(x_{\tau_m},\ldots,x_{\tau_m+k};
        r_{\tau_m},\ldots,r_{\tau_m+k-1}),\\
        \widetilde H_m^k
        &=
        (x_{\tau_m}-p_m,\ldots,x_{\tau_m+k}-p_m;
        r_{\tau_m},\ldots,r_{\tau_m+k-1}).
\end{aligned}
\]
The response lists are empty when $k=0$, so $\widetilde H_m^0=H_*$.  At
global time $t=\tau_m+k$, set
\[
        \widetilde u_k=\mu_m(\widetilde H_m^k),
        \qquad
        u_{\tau_m+k}=p_m+\widetilde u_k.
\]
The successor relation gives
\[
        x_{\tau_m+k+1}-p_m
        =u_{\tau_m+k}+r_{\tau_m+k}-p_m
        =\widetilde u_k+r_{\tau_m+k}.
\]
Hence the translated stage is a play from $H_*$ controlled by $\mu_m$.

Continue this recursion until the stopping index
\[
        k_m^*
        \triangleq
        \inf\bigl\{k\ge1\mid
        \Env_{\ell_m}(\widetilde H_m^k)>B_m\bigr\}.
\]
If $k_m^*<\infty$, set
$\tau_{m+1}=\tau_m+k_m^*$, so stage $m+1$ starts from the current state
$x_{\tau_{m+1}}$.  If $k_m^*=\infty$, stage $m$ continues forever.  Write
$\widetilde\varpi_m$ for its infinite translated play.

These rules define $\mu^*$ recursively on histories consistent with the stage
construction.  The stage decomposition is unique because a new stage begins at
the first budget violation.  Set $\mu^*(H)=0$ on all other histories.  Thus
$\mu^*$ is a causal controller strategy, and a restart changes only its internal
stage data, not the global response history.

We use the following stage boundedness fact.  If stage $m$ is infinite, then
\[
        \Env_{\ell_m}^\infty(\widetilde\varpi_m)\le B_m.
\]
Complete the responses observed
along this stage to an adversary strategy by setting
\[
        \sigma_m(\widetilde H_m^k,\widetilde u_k)=r_{\tau_m+k},
        \qquad k\ge0,
\]
and assigning arbitrary values elsewhere.  The visited histories have distinct
lengths, so this prescription is consistent, and the resulting play is
$\widetilde\varpi_m$.  Since $\mu_m$ wins against $\Esc_{\ell_m,H_*}$, this play
lies outside the escape objective.  Its finite envelope therefore forces its
translated state sequence to be bounded.

\medskip
\noindent\textit{Step 2. Reduction to stage $m_0$.}
Let $\varpi$ denote any play generated by the constructed strategy.  Assume that
\[
        \Env_\alpha^\infty(\varpi)\le A<\infty
        \qquad\text{for some }\alpha<\Gamma_d.
\]
By the choice of the design slopes and since $B_m=m\to\infty$, there exists an
index $m_0$ such that
\[
        \ell_m>\alpha,\qquad B_m>A
        \qquad\text{for all }m\ge m_0.
\]
If some stage $m<m_0$ is infinite, then
\[
        \Env_{\ell_m}^\infty(\widetilde\varpi_m)\le B_m<\infty.
\]
The stage boundedness fact shows that its translated state sequence is bounded.
Since $x_{\tau_m+k}=p_m+(x_{\tau_m+k}-p_m)$, the original states in that stage
are bounded as well.  The preceding stages contain only finitely many states,
so the full state sequence in $\varpi$ is bounded.  It remains to consider the case
in which all stages $m<m_0$ terminate and stage $m_0$ is reached.

\medskip
\noindent\textit{Step 3. The stage-$m_0$ budget is never exceeded.}
For every $k$ reached during stage $m_0$, translation preserves all state and
response differences.  Since $G_{m_0}^k$ is a subhistory of
$H^{\tau_{m_0}+k}$ and $\ell_{m_0}>\alpha$,
\[
\begin{aligned}
        \Env_{\ell_{m_0}}(\widetilde H_{m_0}^k)
        &=\Env_{\ell_{m_0}}(G_{m_0}^k)\\
        &\le\Env_\alpha(G_{m_0}^k)\\
        &\le\Env_\alpha(H^{\tau_{m_0}+k})\\
        &\le A<B_{m_0}.
\end{aligned}
\]
Therefore the stopping condition is never met.  Hence, if stage $m_0$ is
reached, the controller remains there forever.

\medskip
\noindent\textit{Step 4. Boundedness of the state sequence.}
If Step~2 has not already proved boundedness, then stage $m_0$ is reached and
is infinite by Step~3.  In this case,
\[
        \Env_{\ell_{m_0}}^\infty(\widetilde\varpi_{m_0})
        \le B_{m_0}<\infty.
\]
The stage boundedness fact shows that its translated state sequence is bounded.
Adding the fixed vector $p_{m_0}$ proves that the original states in the
infinite stage are bounded.  The preceding stages contain only finitely many
states, so the full state sequence in $\varpi$ is bounded.
\end{proof}

\section{Realization in a Hilbert space}\label{sec:criticality}
This appendix proves the realization lemma used in the supercritical part of
Theorem~\ref{thm:nf-critical}.  The lemma represents a response play with finite
envelope as a plant trajectory after embedding the visited states in a Hilbert
space and choosing a bounded disturbance sequence.

\begin{lemma}\label{lem:nf-realization}
Assume that $\ell\ge0$ and that a play
$\varpi=(H^k,u_k,r_k,H^{k+1})_{k\ge0}$ starts from a history of length zero
$H^0=(x_0)$ and satisfies
\begin{equation}\label{eq:nf-realization-assumption}
        \sup_k \Env_\ell(H^k)\le A<\infty.
\end{equation}
Write $H^k=(x_0,\ldots,x_k;r_0,\ldots,r_{k-1})$.  For $k=0$ the response list
is empty.  Then for every $K>\ell$ there exist $f\colon\R^d\to\R^d$ and a sequence
$(w_t)_{t\ge1}$ such that $f$ is $K$-Lipschitz, $(w_t)_{t\ge1}$ is bounded,
and
\[
        r_i=f(x_i)+w_{i+1},
        \qquad i\ge 0.
\]
Consequently, with the same controls $u_i$, the play's state sequence is
realized by \eqref{eq:plant}, and $\|f\|_{\GL,d}\le K$.
\end{lemma}

\begin{proof}
The construction has three steps.  The finite envelope first gives a pairwise
response bound with additive constant $A$.  Coordinates in an auxiliary Hilbert
space then turn that additive constant into a fixed geometric separation.  After a
Lipschitz extension in the lifted space, pulling the map back to the original
slice leaves a bounded disturbance.
Since $\Env_\ell(H^k)\ge0$ for every $k$, the assumption implies $A\ge0$.

\medskip
\noindent\textit{Step 1. Pairwise response bound.}
The assumption \eqref{eq:nf-realization-assumption} implies
\begin{equation}\label{eq:nf-pairwise}
        \|r_i-r_j\|\le \ell \|x_i-x_j\|+A,
        \qquad i,j\ge 0.
\end{equation}
For every pair $i,j$, choose $k\ge \max\{i,j\}+1$.  Then both responses
$r_i,r_j$ and the corresponding states $x_i,x_j$ appear in the finite response history
$H^k$.  By the definition of $\Env_\ell(H^k)$,
\[
        \bigl(\|r_i-r_j\|-\ell\|x_i-x_j\|\bigr)_+
        \le \Env_\ell(H^k)
        \le A.
\]
This gives \eqref{eq:nf-pairwise}.

\medskip
\noindent\textit{Step 2. Separation in a Hilbert space and extension.}
Let $\ell_2(\{0,1,2,\ldots\})$ denote the Hilbert space of square-summable
real sequences indexed by the nonnegative integers.  Set
$\mathcal H=\R^d\oplus \ell_2(\{0,1,2,\ldots\})$, and let
$(\eta_i)_{i\ge0}$ be the standard orthonormal basis of the second factor.
Since $K>\ell$, we have $K^2-\ell^2>0$.  Choose any
\[
        \rho>0,
        \qquad
        \rho\ge \frac{A}{\sqrt{2(K^2-\ell^2)}}.
\]
Define the lifted points $\xi_i=(x_i,\rho \eta_i)\in\mathcal H$.
For every $s\ge 0$ one has
\[
\begin{aligned}
        (\ell s+A)^2-K^2s^2
        &=-(K^2-\ell^2)s^2+2\ell As+A^2\\
        &=-(K^2-\ell^2)
          \left(s-\frac{\ell A}{K^2-\ell^2}\right)^2
          +\frac{A^2K^2}{K^2-\ell^2}\\
        &\le 2K^2\rho^2.
\end{aligned}
\]
The last inequality is exactly the imposed lower bound on $\rho$.  Since both
sides below are nonnegative, it follows that
\[
        \ell s+A\le K\sqrt{s^2+2\rho^2}.
\]

For $i\ne j$, the orthonormality of $\eta_i$ and $\eta_j$ gives
\[
        \|\xi_i-\xi_j\|^2
        =
        \|x_i-x_j\|^2+\rho^2\|\eta_i-\eta_j\|^2
        =
        \|x_i-x_j\|^2+2\rho^2.
\]
Substituting $s=\|x_i-x_j\|$ and using \eqref{eq:nf-pairwise} therefore gives
\[
        \|r_i-r_j\|
        \le
        K\sqrt{\|x_i-x_j\|^2+2\rho^2}
        =
        K\|\xi_i-\xi_j\|.
\]
For $i=j$ the Lipschitz inequality has both sides equal to zero.  Since
$\rho>0$ and the vectors $\eta_i$ are distinct, the lifted points $\xi_i$ are
distinct.  Hence the rule $G(\xi_i)=r_i$ is well defined and satisfies
\[
        \|G(\xi_i)-G(\xi_j)\|
        \le
        K\|\xi_i-\xi_j\|,
        \qquad i,j\ge0.
\]

We use the Kirszbraun--Valentine extension theorem in the following form for
Hilbert spaces: a $K$-Lipschitz map from a subset of one Hilbert space into
another Hilbert space extends to a $K$-Lipschitz map on the whole domain
\cite{Kirszbraun1934,Valentine1945}.  Since $\R^d$ is a Hilbert space, $G$
extends to a $K$-Lipschitz map
$\widetilde G\colon\mathcal H\to\R^d$.

\medskip
\noindent\textit{Step 3. Pullback and disturbance bound.}
Define
\[
        f(x)=\widetilde G((x,0)),
        \qquad
        w_{i+1}=r_i-f(x_i).
\]
For all $x,y\in\R^d$,
\[
        \|f(x)-f(y)\|
        \le
        K\|(x,0)-(y,0)\|_{\mathcal H}
        =
        K\|x-y\|.
\]
Thus $f$ is $K$-Lipschitz, and Lemma~\ref{lem:glnorm-envelope} gives
$\|f\|_{\GL,d}\le K$.  Moreover, for every $i\ge0$,
\[
        \|w_{i+1}\|
        =
        \|\widetilde G(\xi_i)-\widetilde G((x_i,0))\|
        \le
        K\|\xi_i-(x_i,0)\|_{\mathcal H}
        =
        K\rho.
\]
Thus $(w_t)$ is bounded.  By the definition of $w_{i+1}$, the successor
relation becomes
\[
        x_{i+1}=u_i+r_i=u_i+f(x_i)+w_{i+1},
\]
which is exactly the plant dynamics \eqref{eq:plant}.  Hence the
lifted construction realizes the same state and response sequence by a plant
with generalized Lipschitz seminorm at most $K$ and bounded disturbances.
\end{proof}

\section{Borel determinacy for the escape game}\label{app:martin}
This appendix isolates the determinacy input used in the main text.  The
external result is Martin's Borel determinacy theorem.  We verify, in the
notation of this paper, that the escape game is an alternating infinite game
with a Borel winning set.

\subsection{Abstract Borel determinacy}

\begin{definition}
Let $Y$ be the set of possible moves, and write $Y^\omega$ for the set of
infinite sequences $y=(y_0,y_1,\ldots)$ with entries in $Y$.  For a target set
$B\subseteq Y^\omega$, the Gale--Stewart game $G(B,Y)$ is the alternating game
in which Player~I chooses the even-indexed moves $y_0,y_2,\ldots$ and
Player~II chooses the odd-indexed moves $y_1,y_3,\ldots$.  Player~I wins if and
only if the resulting sequence lies in $B$.
\end{definition}

For a topological space $X$, let $\mathcal B(X)$ denote its Borel
$\sigma$-algebra. For the application with moves in $\R^d$, define
the coordinate projections, for $k\ge0$, by
\[
        p_k\colon(\R^d)^\omega\to\R^d,
        \qquad
        p_k(y_0,y_1,\ldots)=y_k .
\]
The product $\sigma$-algebra is
\[
        \mathcal B(\R^d)^{\otimes\omega}
        \triangleq
        \sigma\bigl(\{p_k^{-1}(E)\mid k\ge0,\ E\in\mathcal B(\R^d)\}\bigr),
\]
where $\sigma(\cdot)$ denotes the generated $\sigma$-algebra and
$p_k^{-1}(E)=\{y\in(\R^d)^\omega:y_k\in E\}$ is a cylinder set.
This is the Borel $\sigma$-algebra of the Euclidean product topology
on $(\R^d)^\omega$.

\begin{definition}
In the game $G(B,Y)$, a strategy records what a player chooses
after each finite list of previous moves.  Let $Y^0=\{\emptyset\}$, and for
$m\ge1$ let
\[
        Y^m
        \triangleq
        \{(y_0,\ldots,y_{m-1})\mid y_i\in Y,\ 0\le i<m\}.
\]
Player~I moves after an even number of previous moves, and Player~II moves
after an odd number. Their strategies are maps
\[
        \sigma_I\colon\bigcup_{n\ge0}Y^{2n}\to Y,
        \qquad
        \sigma_{II}\colon\bigcup_{n\ge0}Y^{2n+1}\to Y.
\]
These maps are not required to be measurable. A pair of strategies determines one sequence
$y\in Y^\omega$ by alternating the two rules.  A Player~I strategy is winning
if, against every Player~II strategy, the resulting sequence belongs to $B$.
A Player~II strategy is winning if, against every Player~I strategy, the
resulting sequence does not belong to $B$.  The game is determined if one of
the two players has a winning strategy.
\end{definition}

For the tree formulation, let
\[
        Y^{<\omega}\triangleq\bigcup_{m=0}^{\infty}Y^m.
\]
A tree on $Y$ is a nonempty set $T\subseteq Y^{<\omega}$ that contains every
initial segment of each of its elements.  It is pruned if every element of $T$
is an initial segment of a longer element of $T$.  Its branch space is
\[
        [T]\triangleq
        \{y\in Y^\omega\mid (y_0,\ldots,y_{m-1})\in T
        \text{ for every }m\ge1\}.
\]
For $B\subseteq[T]$, the tree game $G_T(B)$ is the alternating game in which
the players construct a branch of $T$, and Player~I wins exactly when the
resulting branch belongs to $B$.

We use the tree formulation of Martin's Borel determinacy theorem
\cite{Martin1975}.  Section~20.B of \cite{Kechris1995} equips $Y^\omega$ with
the product topology when an arbitrary set $Y$ is given the discrete topology.
In this notation, \cite[Theorem~20.5]{Kechris1995} states that $G_T(B)$ is
determined whenever $T$ is a nonempty pruned tree on $Y$ and $B\subseteq[T]$ is
Borel in the relative topology.  Thus the theorem imposes neither countability
nor Polish structure on $Y$, and real moves need not be coded by integers.

\begin{theorem}\label{thm:martin-rd}
Fix $d\ge1$ and let $B\in\mathcal B(\R^d)^{\otimes\omega}$.
Then the Gale--Stewart game $G(B,\R^d)$ is determined.
\end{theorem}

\begin{proof}
Set $Y=\R^d$ and $T=Y^{<\omega}$.
Thus $T$ is the set of all finite sequences of elements of $Y$.  This tree is
nonempty and pruned, because every finite sequence can be extended by appending
one more element of $Y$.  Because $T$ contains every finite sequence of
elements of $Y$, its branch space is $[T]=Y^\omega$.  Thus the tree game with target set $B$ has the same
moves and the same winner as $G(B,Y)$.

It remains to compare the two Borel structures.  Let $\R^d_{\mathrm{disc}}$
denote $\R^d$ with the discrete topology used in
\cite[Theorem~20.5]{Kechris1995}.  Every Borel subset $E\subseteq\R^d$ is open
in $\R^d_{\mathrm{disc}}$.  Hence each cylinder
$p_k^{-1}(E)$ is open in $(\R^d_{\mathrm{disc}})^\omega$, and therefore
\[
        \mathcal B(\R^d)^{\otimes\omega}
        \subseteq
        \mathcal B\bigl((\R^d_{\mathrm{disc}})^\omega\bigr).
\]
Thus every admissible target set $B$ is Borel in the topology required by the
cited theorem.

The cited theorem therefore gives determinacy of $G(B,\R^d)$.
\end{proof}

\subsection{The escape game as a Borel game}\label{app:escape-game-application}

\begin{proof}[Proof of Proposition~\ref{prop:escape-determinacy}]
We identify the escape game with the abstract Gale--Stewart game from
Theorem~\ref{thm:martin-rd}.  Fix a finite response history $H^0$ and a slope
$\ell\ge0$.  The proof checks the coding of plays, the coding of strategies,
and the Borel measurability of the coded escape set.

Set $Y=\R^d$ and use the Euclidean Borel $\sigma$-algebra
$\mathcal B(\R^d)$.
Write the fixed initial history as
\[
        H^0=(\bar x_0,\ldots,\bar x_q;\bar r_0,\ldots,\bar r_{q-1}),
\]
where $q\ge0$ and the response list is empty when $q=0$.

First encode plays by a sequence
$y=(y_0,y_1,\ldots)\in(\R^d)^\omega$, and set
\[
        u_t=y_{2t},
        \qquad
        \widehat r_t=y_{2t+1}.
\]
Starting from $H^0$, these moves generate a unique play $H^0,H^1,H^2,\ldots$ by
\[
        H^{t+1}=(H^t)^+(u_t,\widehat r_t).
\]
Let $\Pi_{H^0}(y)$ denote this play.  Conversely, once $H^0$ is fixed, the
generated histories determine the same sequence $y$.  Thus $y$ is a one-to-one
coding of the original play from $H^0$.

Next compare strategies.  At the even move $y_{2t}$, Player~I has seen
$y_0,\ldots,y_{2t-1}$, from which the current response history $H^t$ is
reconstructed.  Conversely, because the initial history is fixed, any finite
response history extending $H^0$ determines the previous moves.  Here
``extends'' means that
the first $q+1$ states and the first $q$ responses agree with $H^0$.  If such a
history has state list $(z_0,\ldots,z_{q+t})$, then its last $t$ responses are
$\widehat r_0,\ldots,\widehat r_{t-1}$.  Its previous moves are recovered from
\[
        y_{2i}=z_{q+i+1}-\widehat r_i,
        \qquad
        y_{2i+1}=\widehat r_i,
        \qquad 0\le i<t.
\]
Indeed, $y_{2i}=u_i$ and the successor relation is
$z_{q+i+1}=u_i+\widehat r_i$.

Thus a controller strategy $\mu$ gives a Player~I strategy by
\[
        \sigma_I(y_0,\ldots,y_{2t-1})=\mu(H^t).
\]
Conversely, a Player~I strategy $\sigma_I$ gives a controller strategy on
histories extending $H^0$ as follows.  If $H$ extends $H^0$ and has length
$q+t$, reconstruct the previous move sequence $y_0,\ldots,y_{2t-1}$ from $H$
and set
\[
        \mu(H)=\sigma_I(y_0,\ldots,y_{2t-1}).
\]
Extending this prescription by the value $0$ on the remaining histories gives a
total causal controller strategy.

The adversary side is analogous.  At the odd move $y_{2t+1}$, Player~II has
seen $y_0,\ldots,y_{2t}$, which determines and is determined by $(H^t,u_t)$.
Hence an adversary strategy $\sigma$ gives a Player~II strategy by
\[
        \sigma_{II}(y_0,\ldots,y_{2t})
        =
        \sigma(H^t,u_t).
\]
Conversely, a Player~II strategy $\sigma_{II}$ defines an adversary strategy on
pairs $(H,u)$ with $H$ extending $H^0$ as follows.  If $H$ has length $q+t$,
reconstruct the previous move sequence $y_0,\ldots,y_{2t-1}$ from $H$ and set
\[
        \sigma(H,u)=\sigma_{II}(y_0,\ldots,y_{2t-1},u).
\]
Extending this prescription by the value $0$ on the remaining pairs gives a
total adversary strategy.
Therefore the abstract strategies and the strategies in the escape game induce the
same plays from $H^0$.

Define
\[
        D_{H^0,\ell}
        \triangleq
        \{y\in(\R^d)^\omega\mid\Pi_{H^0}(y)\in\Esc_{\ell,H^0}\},
        \qquad
        C_{H^0,\ell}\triangleq(\R^d)^\omega\setminus D_{H^0,\ell}.
\]
The set $C_{H^0,\ell}$ is the controller target, because Player~I in
$G(C_{H^0,\ell},\R^d)$ is the first mover and hence
represents the controller.

It remains to check that $D_{H^0,\ell}$ is Borel.  For each fixed $t$, the
finite response history $H^t(y)$ is obtained from $H^0$ and
$y_0,\ldots,y_{2t-1}$ by finitely many additions.  Therefore
$\Env_\ell(H^t(y))$ and $\Rad(H^t(y))$ are maxima of finitely many continuous
functions of the finitely many coordinates $y_0,\ldots,y_{2t-1}$.  A finite
maximum of continuous functions is continuous, and the sets where a continuous
function is at most or at least a fixed number are Borel.  Hence, for each
integer $N\ge1$ and each integer $m\ge1$, the sets
\[
        E_{t,N}\triangleq\{y\in(\R^d)^\omega\mid\Env_\ell(H^t(y))\le N\},
        \qquad
        R_{t,m}\triangleq\{y\in(\R^d)^\omega\mid\Rad(H^t(y))\ge m\}
\]
are Borel subsets of $(\R^d)^\omega$.

The set of sequences with finite envelope is
\[
        \{y\in(\R^d)^\omega\mid
        \Env_\ell^\infty(\Pi_{H^0}(y))<\infty\}
        =
        \bigcup_{N=1}^{\infty}\bigcap_{t=0}^{\infty}E_{t,N}.
\]
The set of sequences whose state radius is unbounded is
\[
        \{y\in(\R^d)^\omega\mid\sup_t\Rad(H^t(y))=\infty\}
        =
        \bigcap_{m=1}^{\infty}\bigcup_{t=0}^{\infty}R_{t,m}.
\]
Therefore the coded escape set is
\[
        D_{H^0,\ell}
        =
        \left(\bigcup_{N=1}^{\infty}\bigcap_{t=0}^{\infty}E_{t,N}\right)
        \cap
        \left(\bigcap_{m=1}^{\infty}\bigcup_{t=0}^{\infty}R_{t,m}\right),
\]
which is Borel because Borel sets are closed under countable unions and
countable intersections.  Hence $C_{H^0,\ell}$ is Borel as well.

Theorem~\ref{thm:martin-rd} applies to
$G(C_{H^0,\ell},\R^d)$.  If Player~I wins, the controller
forces the play outside $\Esc_{\ell,H^0}$.  If Player~II wins, the adversary
forces $\Esc_{\ell,H^0}$.  The two alternatives are mutually exclusive, because
any fixed pair of strategies determines one play.  This proves the stated
determinacy assertion.
\end{proof}

\section{Independence of the initial response history}
\label{app:finite-initial-histories}
This appendix proves Proposition~\ref{prop:finite-history-invariance}.

\begin{proof}
Fix $\ell\ge0$.  The conclusion is immediate if
$\mathcal U_\ell=\varnothing$.  Otherwise, choose
\[
        H'=(y_0,\ldots,y_m;q_0,\ldots,q_{m-1})\in\mathcal U_\ell,
\]
and let
\[
        H=(x_0,\ldots,x_n;r_0,\ldots,r_{n-1})
\]
be an arbitrary finite response history.  Let $\sigma'$ be an adversary
strategy that wins from $H'$ for $\Esc_{\ell,H'}$, and set $c=y_m-x_n$.

For a finite extension $K$ of $H$, let $j\ge0$ be the number of added
responses, set $X_0=x_n$, and write the added states and responses as
\[
        X_1,\ldots,X_j
        \qquad\text{and}\qquad
        a_0,\ldots,a_{j-1},
\]
respectively.  Both lists are empty when $j=0$.  Define $\Phi(K)$ to be the
extension of $H'$ by the states
\[
        Y_k\triangleq X_k+c,\qquad 1\le k\le j,
\]
and the same response list.  Thus $Y_0=y_m$.  Whenever $K$
extends $H$, define
\[
        \sigma_H(K,u)\triangleq\sigma'(\Phi(K),u+c),
\]
and define $\sigma_H$ arbitrarily on all other histories.  This defines a
causal adversary strategy that depends only on $H$, $H'$, and $\sigma'$.

Fix an arbitrary controller strategy $\mu$.  On the image of $\Phi$, define
\[
        \mu'(\Phi(K))\triangleq\mu(K)+c,
\]
and define $\mu'$ arbitrarily elsewhere.  The map $\Phi$ is injective, so
$\mu'$ is well defined.  These definitions show inductively that the plays
generated by $(\mu,\sigma_H)$ from $H$ and by $(\mu',\sigma')$ from $H'$ have
the same appended responses and satisfy
\[
        Y_j=X_j+c,\qquad j\ge0.
\]

Let $\varpi'$ denote the simulated play.  Since
$\varpi'\in\Esc_{\ell,H'}$, its radius is unbounded and
$B'\triangleq\Env_\ell^\infty(\varpi')<\infty$.  The initial history $H'$ is
finite, so the states $Y_j$ have unbounded distance from $y_0$.  Since
\[
        \|X_j-x_0\|
        \ge \|Y_j-y_0\|-\|x_0+c-y_0\|,
\]
the radius of the actual play is also unbounded.  For all $j,k\ge0$, the
responses in the appended tail satisfy
\[
        \|a_j-a_k\|
        \le \ell\|Y_j-Y_k\|+B'
        =\ell\|X_j-X_k\|+B'.
\]
Pairs of responses contained in $H$ satisfy the envelope bound
$\Env_\ell(H)$.  If $n\ge1$, then, for $0\le i<n$ and $j\ge0$,
\begin{align*}
        \|a_j-r_i\|
        &\le \|a_j-a_0\|+\|a_0-r_i\|\\
        &\le \ell\|X_j-x_i\|
        +B'+\|a_0-r_i\|+\ell\|x_i-X_0\|.
\end{align*}
Thus all response pairs in the actual play satisfy the envelope bound with
the finite constant
\[
        B\triangleq
        \max\left\{
        B',\ \Env_\ell(H),
        \max_{0\le i<n}
        \bigl(B'+\|a_0-r_i\|+\ell\|x_i-X_0\|\bigr)
        \right\}.
\]
This constant may depend on the resulting play, as permitted by the definition
of the escape objective.
If $n=0$, the history $H$ contains no response, and one may take $B=B'$.
Thus the actual play belongs to $\Esc_{\ell,H}$.  Since $\mu$ was arbitrary,
$\sigma_H$ wins from $H$, and hence $H\in\mathcal U_\ell$.

Since $H$ was arbitrary, every nonempty $\mathcal U_\ell$ contains all finite
response histories.  This proves
Proposition~\ref{prop:finite-history-invariance}.
\end{proof}

\section{The critical endpoint}\label{app:endpoint}
We prove Theorem~\ref{thm:critical-endpoint} using games with zero envelope
and separated query points.

Write $e_1=(1,0,\ldots,0)$ and
$H_1=(0,e_1;0)$.  The datum $(0,0)$ is called the anchor.
For a query $x$, control $u$, and response $r$, write $x^+=u+r$.
A continuation has envelope at most $B$ at slope $L$ if every pair of
recorded responses satisfies
\begin{equation}\label{eq:endpoint-budget}
 \|r_i-r_j\|\le L\|x_i-x_j\|+B.
\end{equation}
In the auxiliary games below, an adversary with no legal response loses.

\begin{lemma}\label{lem:closed-survival}
Fix a nonnegative integer $n$ and a control bound $0\le M<\infty$. Let
$\mathscr V_{n,M}(L,B,\delta)$ be the histories containing the anchor from
which the adversary can maintain \eqref{eq:endpoint-budget} and mutual
query distances at least $\delta$ for $n$ more moves against all controls
of norm at most $M$. The distance test includes the current query and
the anchor. At each fixed history length these sets are closed
jointly in the history and in $L,B,\delta\ge0$.
An adversary can maintain these constraints forever against unrestricted
controls if and only if the initial history belongs to
$\bigcap_{m=1}^{\infty}\mathscr V_{m,m}(L,B,\delta)$.
The closedness assertion also holds when the distance test is replaced
by a fixed lower bound on all query norms, with only the distinguished
anchor entry excluded.
\end{lemma}

\begin{proof}
We prove closedness by induction on $n$. At horizon zero, the tests are
closed inequalities. For $n\ge1$, a history satisfying the current tests
belongs to $\mathscr V_{n,M}$ exactly when every control $u$ with
$\|u\|\le M$ admits a response $r$ whose successor belongs to
$\mathscr V_{n-1,M}$.
For a fixed $u$, the anchor gives $\|r\|\le L\|x\|+B$, so responses
along a convergent sequence of histories and parameters have a convergent
subsequence. At the limiting parameters, the induction hypothesis places
the successor in $\mathscr V_{n-1,M}$. Thus the set of histories admitting such
a response is closed. Intersecting over all $\|u\|\le M$ proves the
assertion for either distance test.

Now suppose a history belongs to
$\bigcap_{m=1}^{\infty}\mathscr V_{m,m}(L,B,\delta)$, and fix a control $u$.
For integers
$m\ge\max\{1,\|u\|\}$, responses leading to
$\mathscr V_{m-1,m}$ form nonempty nested compact sets.
Cantor's intersection theorem supplies a common response.
Its successor belongs to $\mathscr V_{k,k}$ for every $k$, by taking
$m\ge\max\{k+1,\|u\|\}$ and decreasing both horizon and control bound.
Choosing such a response after each control gives a strategy that
maintains the constraints forever. Conversely, any such strategy
satisfies every finite horizon condition.
\end{proof}

\begin{lemma}\label{lem:finite-contraction}
Let $L\ge0$ and $0\le M<\infty$. Suppose that, from $H_1$ with zero envelope,
a controller using controls of norm at most $M$ forces a query of norm at most
$q$, where $0\le q<1$, within $N<\infty$ moves.
Then there is a single causal controller such that every generated play
from $H_1$ with finite envelope at slope $L$ has a bounded state sequence.
\end{lemma}

\begin{proof}
Choose $\rho$ with $q<\rho<1$. At zero envelope the adversary
cannot keep query norms at least $\rho$ for $N$ moves against controls
bounded by $M$. By closedness with respect to the envelope in the norm
version of Lemma~\ref{lem:closed-survival}, this remains impossible for
some envelope $\beta>0$. Backward induction in the finite recursion
therefore gives a controller that reaches norm less than $\rho$ within
$N$ moves whenever the envelope is at most $\beta$.
Its transient norms are uniformly bounded, since
$\|x^+\|\le M+L\|x\|+\beta$.

Given an envelope bound $B>0$, start a stage when the current query has
norm $s\ge B/\beta$,
rotate and scale the current query to $e_1$, and retain only the anchor.
The normalized envelope is at most $B/s\le\beta$, so the next stage starts
at norm less than $\rho s$. All intermediate query norms are at most
a fixed multiple of $s$. When the current norm is less than $B/\beta$,
one step with zero control has successor norm at most $LB/\beta+B$.
These bounds control both the stage starts and the intermediate states.
For an unknown envelope, start with a positive guess $B$, double it
until it covers the observed envelope whenever the guess is exceeded,
and restart the stage. On a play with finite envelope there are only
finitely many increases of the guess, after which the preceding bounds apply.
\end{proof}

\begin{lemma}\label{lem:winning-separation}
Let $L>0$.  If the adversary wins the original escape game at slope $L$ from $H_1$,
then, for every $0<\delta<\min\{1,1/L\}$, it has a strategy from $H_1$
with zero envelope and mutual query distances at least $\delta$ forever.
\end{lemma}

\begin{proof}
Suppose the stated separated strategy does not exist.
Lemma~\ref{lem:closed-survival} then gives a horizon $N$ and a finite
control bound $M$ for which the adversary cannot maintain separation
at zero envelope. Backward induction in the finite recursion supplies
a bounded controller forcing a query $x$ within distance $\delta$ of an
earlier query point $x_i$ with recorded response $r_i$, within $N$ moves.
The anchor is included among those sites. One further control $-r_i$
gives $\|x^+\|=\|r-r_i\|<L\delta<1$.
The anchor bound $\|r\|\le L\|x\|$ bounds all intermediate query norms and
this last control uniformly. Lemma~\ref{lem:finite-contraction} then
gives a controller preventing escape at every finite envelope,
contradicting the hypothesis.
\end{proof}

\begin{proof}[Proof of Theorem~\ref{thm:critical-endpoint}]
Corollary~\ref{cor:universal-interval} gives
$1<\Gamma_d\le\frac32+\sqrt2<3$.
Choose slopes $L_j\downarrow\Gamma_d$ with $\Gamma_d<L_j<3$.
At each $L_j$ the adversary wins by determinacy and
\eqref{eq:Gamma}.  Lemma~\ref{lem:winning-separation}, with
$\delta=1/3$, gives a strategy with zero envelope and separation at
least $1/3$. For every integer $m\ge1$, closedness in
Lemma~\ref{lem:closed-survival} therefore yields
\[
 H_1\in\mathscr V_{m,m}(\Gamma_d,0,1/3).
\]
The same lemma supplies one strategy maintaining these constraints
forever at slope $\Gamma_d$.  A bounded ball contains only finitely
many points at mutual distances at least $1/3$.  Consequently the norms
of its query points tend to infinity.  In particular, this is a winning
strategy for the original escape objective.  Proposition~\ref{prop:finite-history-invariance}
transfers the win to every finite response history.

Fix a feedback law $\Sigma$ for the plant and an initial state $a$, and put
$b=a-e_1$. Run the separated strategy from $H_1$, indexing its virtual
queries and responses by $z_t,r_t$, $t\ge0$, separately from the anchor
$(0,0)$. Thus $z_0=e_1$. Use virtual controls $v_t=u_t-b$, where $u_t$ is
computed by $\Sigma$ from the physical
history $x_t=b+z_t$ and its past controls. This defines a causal virtual
controller. The zero-envelope data $z_t\mapsto r_t$, including the
anchor $0\mapsto0$, extend by the Kirszbraun--Valentine theorem
\cite{Kirszbraun1934,Valentine1945} to one fixed globally
$\Gamma_d$-Lipschitz map $g$. Set $f(x)=g(x-b)$. Then $f$ is globally
$\Gamma_d$-Lipschitz and, with zero disturbances, its trajectory satisfies
\[
 x_0=a,\qquad x_{t+1}=b+v_t+r_t=u_t+f(x_t).
\]
Since $\|z_t\|\to\infty$, also $\|x_t\|\to\infty$.
The last assertion follows from
Theorem~\ref{thm:nf-critical}(a) and
$\|f\|_{\GL,d}\le\Gamma_d$.
\end{proof}
\section{Dependence on dimension}\label{app:dimension-asymptotic}
This appendix proves Theorem~\ref{thm:dimension-asymptotic}. Throughout,
$C$ denotes an absolute positive constant whose value may increase from
one estimate to the next. Put $\beta=\log d$.

\subsection{Two geometric estimates}
The upper bound uses these estimates to approximate earlier queries and
responses in a subspace of small dimension, then choose a direction with
small inner products against their remaining components.

\begin{lemma}\label{lem:convex-recurrence}
Let $N\ge1$ be an integer and $v_i,z_i\in\R^d$, $1\le i\le N$, satisfy
\[
 v_i=\sum_{j=1}^{i-1}w_{ij}v_j+z_i,
 \qquad w_{ij}\ge0,\qquad\sum_{j=1}^{i-1}w_{ij}\le1.
\]
Put $B=\sum_{i=1}^N\|z_i\|$. For every integer $r\ge1$, a linear subspace
of dimension less than $r$ approximates every $v_i$ with error at most
$\exp(1/2)B/r$.
\end{lemma}

\begin{proof}
Let $W\in\R^{N\times N}$ have entries $w_{ij}$ for $j<i$ and zero
otherwise, and put $T=(I_N-W)^{-1}$. Then $v_i=\sum_jT_{ij}z_j$.
Writing $\mathbf e_i$ for the $i$th coordinate vector in $\R^N$,
each row of $I_N-W$ is a convex combination of
$\mathbf e_i^\top$ and $\mathbf e_i^\top-\mathbf e_j^\top$. Choosing one such row at each
index gives a transposed directed incidence matrix with one vertex
column deleted. Its minors belong to $\{0,\pm1\}$
\cite[Theorem~8.4]{Schrijver2017}.
Multilinearity therefore bounds every minor of $I_N-W$ by one in absolute
value. Since $\det(I_N-W)=1$, Jacobi's complementary minor identity
gives the same bound for every minor of $T$.

Choose any $r$ linearly independent vectors $v_i$. In an orthonormal
basis of their span, let $Z$ have the nonzero projected $z_j$ as columns $Z_j$.
Let $A$ consist of the selected rows of $T$ with the same columns retained,
and put $D=\operatorname{diag}(\|Z_j\|)$, so $\operatorname{tr}D\le B$.
By Cauchy--Binet and Cauchy--Schwarz, their squared volume satisfies
\[
 |\det(AZ^\top)|^2
 \le\det(ADA^\top)\det(ZD^{-1}Z^\top)
 \le\frac{B^r}{r!}\left(\frac Br\right)^r.
\]
Here $\det(ADA^\top)\le\sum_{|J|=r}\prod_{j\in J}D_{jj}\le B^r/r!$
by the minor bound, while $ZD^{-1}Z^\top$ has trace at most $B$.
Since $r!\ge(r/\exp(1))^r$, a greedy procedure cannot select $r$
vectors each at distance greater than $\exp(1/2)B/r$ from the
span of its predecessors. At termination the span has the required
dimension and approximation error.
\end{proof}

The second estimate follows from Ball--Pajor's volume inequality.
\begin{lemma}
For a nonempty finite list $y_i$ in a $p$-dimensional Euclidean space, $p\ge1$,
there is a unit vector $e$ such that
\begin{equation}\label{eq:list-direction}
 \max_i|e\cdot y_i|
 \le\frac{\exp(3/2)}{p^{3/2}}\sum_i\|y_i\|.
\end{equation}
\end{lemma}

\begin{proof}
For $p=1$ this follows from $\max_i\|y_i\|\le\sum_i\|y_i\|$.
If the list spans a proper subspace, choose a perpendicular unit vector.
Otherwise put $S=\sum_i\|y_i\|$,
$P=\{x:|x\cdot y_i|\le1\text{ for every }i\}$, and
$\rho=\min_{\|e\|=1}\max_i|e\cdot y_i|>0$.
Every point of $P$ has norm at most $1/\rho$.
Ball--Pajor's inequality \cite[Theorem~2]{BallPajor1990}, at exponent one, gives
\((p/S)^p\le\operatorname{vol}(P)\le\kappa_p/\rho^p\),
where $\kappa_p$ is the volume of the unit ball. Hence
\[
 \rho\le\kappa_p^{1/p}S/p
 \le\sqrt{2\pi\exp(1)}\,S/p^{3/2}
 <\exp(3/2)S/p^{3/2}.
\]
\end{proof}

\subsection{Upper bound}
\begin{proof}[Proof of Theorem~\ref{thm:dimension-asymptotic}]
We construct a strategy with zero envelope at slope $L>1$ from
$H_1=(0,e_1;0)$.
Number its queries $x_j$, $j\ge1$, and retain the anchor
$(x_0,r_0)=(0,0)$.  Set $A_0=1$ and $\nu_0=e_1$.
Choose $q>1$ and $b>0$ satisfying
\begin{equation}\label{eq:upper-parameters}
 g=q(1-b)>1,\qquad q^2(1+4b)\le L^2,
 \qquad 1-2b\ge q^{-2}.
\end{equation}
At query $j$, put
\[
 K_j=\operatorname{conv}\{x_i:0\le i<j\},\quad
 y_j=\operatorname{proj}_{K_j}x_j,\quad
 \rho_j=\|x_j-y_j\|,\quad A_j=q\rho_j.
\]
By Euclidean Lipschitz extension, followed by projection onto the convex
hull of the previous responses, there is a point
$c_j\in\operatorname{conv}\{r_i:0\le i<j\}$ such that
\[
 \|c_j-r_i\|\le L\|y_j-x_i\|\qquad 0\le i<j.
\]
The projection does not increase distance to any point of that hull.
We will choose a unit vector $\nu_j$ satisfying
\begin{equation}\label{eq:upper-direction}
 |\nu_j\cdot x_i|\le bA_j\quad 0\le i\le j,
 \qquad
 |\nu_j\cdot r_i|\le bA_j\quad 0\le i<j.
\end{equation}
Choose its sign so that $\nu_j\cdot(u_j+c_j)\ge0$, and return
$r_j=c_j+A_j\nu_j$.  Then $\nu_j\cdot x_{j+1}\ge A_j$.
The projections of the previous query points onto $\nu_{j-1}$ are
bounded in absolute value by $bA_{j-1}$. Hence
\begin{equation}\label{eq:upper-growth}
 \rho_j\ge(1-b)A_{j-1},\qquad A_j\ge gA_{j-1}.
\end{equation}
These inequalities hold initially as well.  Projection onto $K_j$ gives
\[
 \|x_j-x_i\|^2\ge\|y_j-x_i\|^2+\rho_j^2.
\]
Since $|\nu_j\cdot c_j|\le bA_j$, every previous response satisfies
\begin{align*}
 \|r_j-r_i\|^2
 &\le L^2\|y_j-x_i\|^2+(1+4b)A_j^2\\
 &\le L^2\|x_j-x_i\|^2.
\end{align*}
Thus all pair constraints hold. Together with \eqref{eq:upper-growth},
the estimate $\nu_j\cdot x_{j+1}\ge A_j$ implies $\|x_j\|\to\infty$.
It remains to verify that \eqref{eq:upper-direction} can always be met
in the claimed dimension, including for arbitrarily large controls.

Define $h_i=x_i-y_i-A_{i-1}\nu_{i-1}$.
The estimate at the preceding query gives
$\nu_{i-1}\cdot(x_i-y_i)\ge(1-b)A_{i-1}$.  Together with $1-2b\ge q^{-2}$ this gives
\begin{equation}\label{eq:residual-energy}
 \|h_i\|^2\le\frac{A_i^2-A_{i-1}^2}{q^2}.
\end{equation}

Let $k,R,r,p$ be positive integers. The integers $k$ and $R$ count recent
directions and residuals removed by orthogonality. The parameter $r$
controls the dimensions used to approximate the remaining vectors, while
$p$ is the dimension available for choosing a direction.
At query $T$, impose orthogonality
to $\nu_i$ for $\max\{0,T-k\}\le i<T$ and to $h_i$ for
$\max\{1,T-R+1\}\le i\le T$. Put $s=\max\{0,T-k\}$
and $J=T-R$. For any convex recurrence through index $j$, divide its
$i$th vector by $\sqrt{A_i A_j}$. The new coefficients are the original
coefficients multiplied by $\sqrt{A_l/A_i}$ for $l<i$, so the recurrence
still satisfies Lemma~\ref{lem:convex-recurrence}. Write
$\zeta=1-g^{-1/2}$. If $s=0$, the earlier query and response lists contain
only zero vectors. For $s>0$, we have $s=T-k$, and the
original innovations have norms $A_i/q$ and $A_i$, respectively.
The sums of rescaled innovation norms are at most $1/(q\zeta)$ and
$1/\zeta$. Applying the lemma to each rescaled list and projecting
perpendicular to the two resulting subspaces bounds the sum of the
projected lengths of the original lists by
\[
 \frac{\exp(1/2)(1+1/q)}{r\zeta^2} A_s
 \le \frac{\exp(1/2)(1+1/q)}{r\zeta^2}g^{-k}A_T.
\]
Here the conversion back to the original lists uses
$\sum_{i\le s}\sqrt{A_i A_s}\le A_s/\zeta$.

After projection onto the orthogonal complement of the specified
directions, each recent response is a convex combination of the earlier
projected responses and zero. Each recent query point has the same
decomposition, with an additional contribution from the residuals.
For $1\le i\le T$, choose weights $w_{il}\ge0$,
$0\le l<i$, with $\sum_{l<i}w_{il}=1$ and
$y_i=\sum_{l<i}w_{il}x_l$.
This additional vector is the projection of $v_i$ from the recursion
\[
 v_i=\sum_{s<l<i} w_{il}v_l+\widetilde h_i,\qquad s<i\le T,
 \qquad
 \widetilde h_i=\begin{cases}h_i,&i\le J,\\0,&i>J.\end{cases}
\]
This follows by substituting $x_i=y_i+A_{i-1}\nu_{i-1}+h_i$ and
using the prescribed orthogonality. The response decomposition follows
by substituting $r_i=c_i+A_i\nu_i$.
If $J>s$, then
\[
 \sum_{s<i\le J}\frac{\|h_i\|}{\sqrt{A_i A_J}}
 \le\frac{2}{q\sqrt{\zeta}} .
\]
Indeed, \eqref{eq:residual-energy} gives
$\|h_i\|^2/A_i\le4\sqrt{A_i}(\sqrt{A_i}-\sqrt{A_{i-1}})/q^2$.
Cauchy--Schwarz, telescoping of the differences, and
$\sum_{i\le J}\sqrt{A_i}\le\sqrt{A_J}/\zeta$ give the stated bound.
The rescaled recurrence for $v_i$ through $J$ therefore gives a third
subspace for which the sum of the projected lengths of the original
vectors is at most $2\exp(1/2)g^{-R}A_T/(qr\zeta^{3/2})$.
Later vectors are convex combinations of these vectors and zero.
When $J\le s$, all modified residuals vanish.

Apply \eqref{eq:list-direction} on a $p$-dimensional subspace
perpendicular to the prescribed directions, residuals, and three
approximation spaces. Its unit direction bounds all tested inner
products by $bA_T/2$ provided that
\begin{equation}\label{eq:upper-choice}
 \exp(2)\left[(1+1/q)\frac{g^{-k}}{\zeta^2}
                  +\frac{2g^{-R}}{q\zeta^{3/2}}\right]
 \le \frac b2 r p^{3/2}.
\end{equation}
The convex combinations above then give \eqref{eq:upper-direction}.
The construction of Lemma~\ref{lem:convex-recurrence} uses at most $r-1$
dimensions for each list. Thus $k+R+3(r-1)+p$ dimensions suffice.

For $\beta\ge17/2$, take
\[
 \begin{aligned}
 a&=\frac{\beta/2+(3/2)\log \beta+16}{d},\qquad
 b=\frac{a}{3\beta},\qquad q=\frac{\exp(a)}{1-b},\\
 L&=q\sqrt{1+4b},\qquad r=\left\lceil\frac{2}{3a}\right\rceil,
 \qquad p=\left\lceil\frac{3}{a}\right\rceil .
 \end{aligned}
\]
Here $g=\exp(a)$, $a<1/40$, and $\zeta\ge10a/21$.
Conditions \eqref{eq:upper-parameters} hold since
$2b\le a\le 1-\exp(-2a)$.
Choose
\[
 k=\left\lceil\frac{\frac12\log(1/a)+\log \beta+11/2}{a}\right\rceil,
 \qquad R=\left\lceil\frac{\log \beta+5}{a}\right\rceil.
\]
The sum of the two terms on the left in \eqref{eq:upper-choice}, divided
by $brp^{3/2}/4$, is bounded by
\[
 4\sqrt3\left[(21/10)^2\exp(-7/2)
              +(21/10)^{3/2}\exp(-3)\right]<2.
\]
Using $\log(1/a)\le \beta-\log(\beta/2)$ in the definitions of $k,R,r,p$ gives
\[
 a[k+R+3(r-1)+p]
 \le ad-\frac12+\frac12\log2+3a<ad.
\]
Here $ad\ge \beta/2$. The resulting strategy gives $\Gamma_d\le L$.
Since $-\log(1-b)+\frac12\log(1+4b)\le3b$, we have
$\log L\le a(1+1/\beta)$. The exponential series gives
$L-1\le a(1+1/\beta)+a^2$. Both $ad/\beta$ and $(ad)^2/d$ are decreasing
for $\beta\ge17/2$. Their values at $17/2$ give
$d(L-1-a)\le ad/\beta+(ad)^2/d<3$.
This proves $d(\Gamma_d-1)\le \beta/2+(3/2)\log \beta+19$ for $\beta\ge17/2$.

For moderate dimensions, we use a second construction. Let $d\ge5$,
$m=\lfloor(d-1)/2\rfloor$, and $q>1$. Suppose the
Lipschitz data with slope $L\ge q$ satisfy $\|r_i\|\le q\|x_i\|$ at
each recorded query point. At a new point $y$ there is a response $c$ in
the convex hull of the recorded responses and zero satisfying
\[
 \|c\|\le q\|y\|,\qquad \|c-r_i\|\le L\|y-x_i\|.
\]
To see this, assign $r_i$ to $(x_i,0)\in\R^{d+1}$ and zero to
$\xi=((1-\lambda^2)y,\lambda\sqrt{1-\lambda^2}\,\|y\|)$,
where $\lambda=q/L$. The identity
\[
 \|(x,0)-\xi\|^2
 =\lambda^2\|x\|^2+(1-\lambda^2)\|x-y\|^2
 \qquad x\in\R^d
\]
shows that the enlarged data are still $L$-Lipschitz.
The Kirszbraun--Valentine theorem \cite{Kirszbraun1934,Valentine1945}
extends them to $(y,0)$, giving the required pair bounds and
$\|c\|\le L\|(y,0)-\xi\|=q\|y\|$.
Projection onto the convex hull of the recorded responses and zero supplies $c$
in that hull while preserving all the ball constraints.

In the windows below, take all available query points or responses if
there are fewer than $m$.
Project $x_j$ onto the convex hull of zero and the last $m$ query points
with recorded responses, with projection $y_j$ and distance $\rho_j$.
Apply the preceding
extension to these data at $y_j$ and return $r_j=c_j+A_j \nu_j$, where
$A_j=q\rho_j$. Choose $\nu_j$ perpendicular to the last $m$ recorded
responses and the last $m$ query points through $x_j$. Such a unit vector
exists because $d\ge2m+1$. Choose its sign so that
$\nu_j\cdot x_{j+1}\ge A_j$, as in the preceding construction.
The vector $\nu_{j-1}$ is perpendicular to the hull used at $j$, so
$\rho_j\ge A_{j-1}$. From the seed $H_1$, the bounds
$A_1=q$ and $A_j\ge qA_{j-1}$ give escape. Projection and
$\nu_j\perp c_j$ give all recent pair constraints and the anchor bound
$\|r_j\|\le q\|x_j\|$.

For old pair constraints, use the bounds $\|x_i\|\le H_q A_i$ and
$\|r_i\|\le C_q A_i$, with $H_q=1/\sqrt{q(q-1)}$ and
$C_q=q/\sqrt{q^2-1}$. These hold by induction. The bound on responses
follows from $\|c_j\|\le C_q A_{j-1}$ and
$C_q^2/q^2+1=C_q^2$. For the query norms, set $v=A_{j-1}/\rho_j\in[0,1]$.
The preceding direction is perpendicular to $y_j$ and the residual
$x_j-y_j$ has component at least $A_{j-1}$ in this direction. Thus
\[
 \|x_j\|^2/\rho_j^2
 \le 1+H_q^2v^2+2H_qv\sqrt{1-v^2}\le q^2H_q^2.
\]
Here Young's inequality gives
$2H_qv\sqrt{1-v^2}\le v^2/q+(1-v^2)/(q-1)$, and
$H_q^2+1/q=1/(q-1)$. The estimates at the first query follow from $y_1=c_1=0$.
For an older site $i<j-m$, we have $q^m A_i\le A_{j-1}\le\|x_j\|$.
If $(L-q)q^m\ge LH_q+C_q$, this gives
\[
 \|r_j-r_i\|\le q\|x_j\|+C_q A_i
 \le L(\|x_j\|-\|x_i\|)\le L\|x_j-x_i\|.
\]
Solving for $L$ gives, for every $q>1$ with $q^m>H_q$,
\begin{equation}\label{eq:finite-window-upper}
 \Gamma_d\le q+\frac{qH_q+C_q}{q^m-H_q}.
\end{equation}
Taking $q=1.12$ for $d=100$ and $q=1.0178$ for $d=1000$ gives
the two numerical bounds following Theorem~\ref{thm:dimension-asymptotic}.

For $\beta<17/2$, the scalar upper bound covers $2\le d\le10$.
For $d\ge11$, put $n=m+1\ge6$ and
$x=\log q=3\log n/(2(n-1))$, and denote the right side of
\eqref{eq:finite-window-upper} by $V$.
Then $x\le3/5$, $q^m=n^{3/2}$, and $H_q\le1/\sqrt x<q^m$.
Using $\exp(x)-1\le x+(7/10)x^2$ and
$\sqrt{x}(qH_q+C_q)\le\sqrt{1+x}+\sqrt{1/2+x}$, with
$nx\ge(3/2)\log n$, gives
\[
\begin{aligned}
 2n(V-1)-3\log n
 &\le \frac{13}{\sqrt n}+
 \frac{2(\sqrt{1+x}+\sqrt{1/2+x})}{\sqrt{(3/2)\log n}-1/n}\\
 &\le \frac{13}{\sqrt n}+\frac5{\sqrt{\log n}}.
\end{aligned}
\]
The first bound uses $n/(n-1)\le6/5$ and
$\log n\le n^\alpha/(\alpha\exp(1))$ for $\alpha=1/2,1/4$.
The second uses $x\le3/5$ and
$\sqrt{3/2}-1/(n\sqrt{\log n})>1$.
Since $d/2\le n\le(d+1)/2$ and $3/d\le d^{-1/2}$, we obtain
$d(\Gamma_d-1)\le3\beta-3\log2+20\exp(-\beta/2)+5/\sqrt{\beta-\log2}$.
The difference between $\beta/2+(3/2)\log \beta+19$ and this bound is concave
on $[\log11,17/2]$ and positive at both endpoints. This finishes the upper estimate.

\subsection{Lower bound}
The lower term $\beta/2-3\log \beta-32$ is immediate from
$\Gamma_d\ge1$ when it is nonpositive. Otherwise $\beta>85$, as follows
by convexity on $[\log 2,85]$. Set
\[
 \begin{aligned}
 \epsilon&=\frac{\beta/2-3\log \beta-32}{d},\quad L=1+\epsilon,
 \quad K=\lceil d(1+1/\beta)\rceil,\\
 a_*&=\exp(-\min\{\log 2,d\epsilon\})\in[1/2,1).
 \end{aligned}
\]
A bounded controller forcing a query of norm at most $a_*$ within $K$
moves from $H_1$ at slope $L$ and zero envelope suffices by
Lemma~\ref{lem:finite-contraction}.

Suppose no contraction occurs within $K$ moves. The loss estimates then
leave many starting indices whose tests pass. A bound on positive
correlations allows the Welch inequality to force a substantial sum of
negative correlations among their successors. The pair constraints
transfer these correlations to the predictor, whose negative drift in
the logarithm of the query norm contradicts the absence of contraction.
The numerical estimates used below are verified in Step~4.

\medskip
\noindent\textit{Step 1. Controller and basic estimates.}
At a stage start, rotate and scale the current query
to have norm one and retain only the anchor $(0,0)$.
Number the stage queries $x_0,x_1,\ldots$, separately from that anchor,
and write $a_i=\|x_i\|$, $z_i=x_i/a_i$ while $a_i>0$.
Stop the stage as soon as $a_i\le a_*$.

At query $j$, let $c_j,R_j$ be the center and radius of the minimum
ball enclosing all responses feasible under the earlier stage data and
the anchor.  The feasible set is compact and nonempty on a legal play,
and lies in the ball of radius $La_j$ about zero.  Its enclosing center
belongs to the convex hull of the feasible set.  Hence
$R_j\le La_j$ and $\|c_j\|\le La_j$.
Set $\tau=12\beta\epsilon/L$ and choose
\begin{equation}\label{eq:lower-control}
 u_j=-c_j/2+L\tau a_jp_j,
 \qquad p_j=\frac{V_j}{\max\{1,\|V_j\|\}}.
\end{equation}
For each completed move with $a_{i+1}>0$, record the observable loss
\begin{equation}\label{eq:observable-loss}
 s_i=\log\frac{La_i(1+\tau p_i\cdot z_{i+1})}{a_{i+1}}.
\end{equation}
To define the predictor, put
\[
 \begin{aligned}
 \bar\epsilon&=\max\{\epsilon,1/d\},\qquad \phi=128\beta\bar\epsilon,
 \qquad n=\lceil \beta\rceil,\qquad N=n+2,\\
 c_*&=1/32,\qquad g_l=\min\{c_*,2^l\phi\}\quad 0\le l<N.
 \end{aligned}
\]
The $N$ thresholds $g_l$ limit the losses after a retained start.
At query $j$, retain $i<j$ if
$s_{i+l}\le g_l$ for every $0\le l<N$ with $i+l<j$, and put
\begin{equation}\label{eq:lower-predictor}
 V_j=\sum_{i\text{ retained}}(-z_i\cdot z_j)_+z_{i+1}.
\end{equation}
This is causal because every tested loss and every $z_{i+1}$ has
already been observed. Initially $V_0=0$.

For every feasible response,
\begin{equation}\label{eq:half-center-identity}
 \|r_i-c_i/2\|^2
 =\frac{\|r_i\|^2+\|r_i-c_i\|^2}{2}-\frac{\|c_i\|^2}{4}
 \le L^2a_i^2.
\end{equation}
Projection onto $z_{i+1}$ proves $s_i\ge0$ and gives
\begin{equation}\label{eq:lower-log-growth}
 \log(a_{i+1}/a_i)
 \le\epsilon+\tau p_i\cdot z_{i+1}-s_i,
 \qquad a_{i+1}\le L(1+\tau)a_i.
\end{equation}
No logarithm is used after the stopping condition is met.

Consider a prefix of $T\le K$ moves before contraction, through a
possible first violation of
\begin{equation}\label{eq:radius-bootstrap}
 a_j/a_i\le2\exp((j-i)\epsilon)\qquad 0\le i<j\le T.
\end{equation}
The one-step bound in \eqref{eq:lower-log-growth} ensures the same
inequality with constant three throughout such a prefix.  With
$E_0=\exp(K\epsilon)$, we thus have
\begin{equation}\label{eq:radius-moments}
 \begin{gathered}
 1/2\le a_i\le3E_0,\qquad
 E_0\le 3\exp(-32)\sqrt d/\beta^3,\\
 \left(\sum_{i=0}^T a_i^p\right)
 \left(\sum_{i=0}^T a_i^{-p}\right)
 \le 200E_0^p/\bar\epsilon^2\qquad p=1,2.
 \end{gathered}
\end{equation}
For the last estimate, put $S_p=\sum_{k=0}^K\exp(-pk\epsilon)$.
For $0\le i<j\le T$, the ratios $a_j/a_i$ and $a_i/a_j$ are at most
$3\exp((j-i)\epsilon)$ and $6\exp(i\epsilon)$, respectively.
Expanding the product and summing gives at most
$(1+3^p+6^p)E_0^pS_p^2$, using
$\sum_{k=1}^T k\exp(-pk\epsilon)\le S_p^2$ for the backward sum.
Since $S_p\le\min\{K+1,2/(p\epsilon)\}\le2/\bar\epsilon$ and
$4(1+3^p+6^p)\le184<200$, the bound follows, also on subsets.

The starts whose $N$ tests are completed and passed in this prefix form
\[
 I=\{i\in\mathbb Z_{\ge0}:i+N\le T,\quad
                    s_{i+l}\le g_l\text{ for every }0\le l<N\}.
\]
The set $I$ is used only in the analysis. For $j\in I$, the retained
indices $i<j$ are exactly the earlier members of $I$. To see this, a test
of $i$ not yet observed at time $j$ has the form $s_{i+l}=s_{j+h}$ with
$0\le h<l<N$. It is covered by the test of $j$ at offset $h$, because
$g_h\le g_l$. All tests of an earlier member of $I$ have passed.
On any suffix $[u,T)$, the number of indices outside $I$ is at most
\begin{equation}\label{eq:bad-start-count}
 N+\frac{3}{\phi}S_{u,T},
 \qquad S_{u,T}=\sum_{u\le i<T}s_i.
\end{equation}
For each offset $l$, a failed test has loss greater than $g_l$.
Summing these indicators uses
$\sum_{l<N}1/g_l\le2/\phi+N/c_*\le3/\phi$.

\medskip
\noindent\textit{Step 2. Correlations and predictor errors.}
Write $I+1=\{i+1:i\in I\}$ for the indices of the successors $z_{i+1}$.
We prove that, for $E=I$ or $E=I+1$,
arbitrary nonnegative weights $v_i$ on $E$, and $C_0=1000$,
\begin{equation}\label{eq:weighted-Gram}
 2\sum_{\substack{i,j\in E\\i<j}}v_iv_j(z_i\cdot z_j)_+
 \le\theta\sum_{i\in E}v_i^2
       +C_0\phi\left(\sum_{i\in E}v_ia_i\right)
                \left(\sum_{i\in E}v_i/a_i\right),
 \qquad\theta=1/3.
\end{equation}
All constants are uniform over the prefix and weights.

Put $\widetilde r_i=r_i/L$,
$\ell_i=a_i^2-(R_i/L)^2$, and
$\delta_i=a_i^2-\|\widetilde r_i\|^2$.
For weights $\omega_i\ge0$ on $i<j$ totaling at most one, assign the
remaining weight to the anchor.  Averaging the squared pair constraints
and completing the square produces a ball containing all feasible
responses.  Comparing its radius with $R_j$ gives
\begin{equation}\label{eq:averaged-ball}
 2\sum_{i<j}\omega_i x_i\cdot x_j
 \le\ell_j+\sum_{i<j}\omega_i\delta_i
                       +\left\|\sum_{i<j}\omega_i\widetilde r_i\right\|^2.
\end{equation}
At a move with loss $s_i\le c_*$, the projection of the unsteered
successor in units $La_i$ is exactly
\[
 1-(1+\tau p_i\cdot z_{i+1})(1-\exp(-s_i))\ge1-2s_i.
\]
Equation~\eqref{eq:half-center-identity} therefore implies
\begin{equation}\label{eq:good-move-geometry}
 \begin{gathered}
 0\le\ell_i,\delta_i\le Cs_i a_i^2,
 \qquad\|c_i\|\le CLa_i\sqrt{s_i},\\
 \left\|\frac{\widetilde r_i}{a_i}
 -(1+\tau p_i\cdot z_{i+1})\exp(-s_i)z_{i+1}\right\|
 \le C\sqrt{s_i}+\tau.
 \end{gathered}
\end{equation}
The coefficient of $z_{i+1}$ lies in $[0,1+\tau]$.
The logarithmic change $\log(a_{i+1}/a_i)$ at such a move lies between
$-c_*-2\tau$ and $\epsilon+\tau$.

The case of all zero weights is immediate.  Otherwise, for a shifted
weighted family on $E$, set, for $0\le l\le n$,
\[
 \begin{aligned}
 A&=\sum_{i\in E}v_i^2,\qquad U_l=\sum_{i\in E}v_ia_{i+l},\qquad
 W_l=\sum_{i\in E}v_i/a_{i+l},\\
 F_l&=2\sum_{\substack{i,j\in E\\i<j}}v_iv_j(z_{i+l}\cdot z_{j+l})_+.
 \end{aligned}
\]
For either choice of $E$, we have $i+l\le T-1$ for $i\in E$ and
$0\le l\le n$, so all shifted vectors in these sums are available in the
prefix. The losses needed at shift $l<n$ are bounded by
$g_{l+1}\le2^{l+1}\phi$ and by $c_*$.
Consequently $U_lW_l\le\exp(4c_*l)U_0W_0$, since
$c_*+\epsilon+3\tau\le4c_*$.
In \eqref{eq:averaged-ball}, at row $j+l$, use
$\omega_{i+l}=\lambda v_i/(a_{i+l}W_l)$ with $\lambda=1/16$
only for $i\in E$, $i<j$, with positive shifted inner product. Multiply
by $v_jW_l/(\lambda a_{j+l})$ and sum over $j\in E$.
The two deficit terms contribute at most $Cg_{l+1}U_lW_l$.
For any selected subset $S\subseteq E$, \eqref{eq:good-move-geometry} gives
\[
 \left\|\sum_{i\in S}v_i\widetilde r_{i+l}/a_{i+l}\right\|^2
 \le2(1+\tau)^2(A+F_{l+1})
             +C(g_{l+1}+\tau^2)\left(\sum_{i\in S}v_i\right)^2.
\]
Here negative cross terms can be discarded in the squared norm of the
successor sum, since its coefficients lie between zero and
$(1+\tau)v_i$.  In the row summation this bound is multiplied by at
most $\lambda$.  Since $(\sum v_i)^2\le U_lW_l$ and $\tau^2\le\phi$,
we obtain
\[
 F_l\le C\phi(2\exp(4c_*))^lU_0W_0
                          +(3/16)(A+F_{l+1}).
\]
The product $(3/16)2\exp(4c_*)$ is less than one.
Also $F_n\le KA$ and $K(3/16)^n<1/12$.
Iteration proves \eqref{eq:weighted-Gram} since $3/13+1/12<1/3$.

Write $h_{ij}=z_i\cdot z_j$ and $b_{ij}=(-h_{ij})_+$.
By \eqref{eq:weighted-Gram} with unit weights and
\eqref{eq:radius-moments}, the positive and negative correlation sums satisfy
\begin{equation}\label{eq:global-Gram-mass}
 \max\left\{\sum_{\substack{i,j\in E\\i<j}}(h_{ij})_+,\,
             \sum_{\substack{i,j\in E\\i<j}}b_{ij}\right\}
 \le C(d+\phi E_0/\bar\epsilon^2),\qquad E=I,\ I+1.
\end{equation}
For the negative sum, use the nonnegativity of the squared norm of
the sum of the unit vectors.
For a row $j\in I$ define
\[
 A_j=\sum_{i\in I}b_{ij}^2,\qquad
 e_j=\sqrt{C_0\phi
       \left(\sum_{i\in I}b_{ij}a_i\right)
       \left(\sum_{i\in I}b_{ij}/a_i\right)}.
\]

For any nonnegative weights $v_i$ on $I$,
$\|\sum_i v_i z_i\|^2\ge0$ and \eqref{eq:weighted-Gram} give
\[
 \sum_{i,j\in I}b_{ij}v_iv_j
 \le(1+\theta)\sum_{i\in I}v_i^2+
 C_0\phi\left(\sum_{i\in I}v_i a_i\right)
       \left(\sum_{i\in I}v_i/a_i\right).
\]
Apply this with $v_i=a_i$ and $v_i=a_i^{-1}$. Cauchy--Schwarz gives
\begin{equation}\label{eq:row-error-moments}
 \begin{aligned}
 \sum_{j\in I}e_j
 &\le\sqrt{C_0\phi}
       \sqrt{\left(\sum_{i,j\in I}b_{ij}a_i a_j\right)
             \left(\sum_{i,j\in I}\frac{b_{ij}}{a_i a_j}\right)}\\
 &\le\sqrt{C_0\phi}\,(1+\theta+C_0\phi|I|)
       \sqrt{\left(\sum_{i\in I}a_i^2\right)
             \left(\sum_{i\in I}a_i^{-2}\right)}.
 \end{aligned}
\end{equation}
Since $|I|\le K\le2d$ and $\phi d\ge128\beta$,
\eqref{eq:radius-moments} with $p=2$ yields
\begin{equation}\label{eq:row-error-total}
 \sum_{j\in I}e_j
 \le C\phi^{3/2}dE_0/\bar\epsilon
 < d/(64\beta).
\end{equation}

Test \eqref{eq:weighted-Gram} with weights $b_{ij}$ for a fixed row
$j$.  The projection of their vector sum on $z_j$ is $-A_j$.
Consequently
\[
 A_j^2\le(1+\theta)A_j+e_j^2,
 \qquad A_j\le1+\theta+e_j.
\]
For $j\in I$, its predictor uses precisely the earlier members of
$I$.  Apply \eqref{eq:weighted-Gram} on $I+1$ with those weights.
Since the ratios $a_{i+1}/a_i$ for $i\in I$ vary by a factor less than
two, it follows that
\[
 \|V_j\|^2\le(1+\theta)A_j+2e_j^2\le4(1+e_j)^2.
\]

\medskip
\noindent\textit{Step 3. Contraction.}
For $i<j$ in $I$ or $I+1$, take the singleton weight
$\omega_i=\sqrt\phi\,a_j/a_i$ in \eqref{eq:averaged-ball}.
It is at most one because $\sqrt\phi E_0<1/100$.
The deficits at these indices are at most $C\phi a_i^2$.
Dividing the resulting inequality by $\omega_i a_i a_j$ yields
\begin{equation}\label{eq:pointwise-positive}
 (h_{ij})_+\le C(\sqrt\phi+\phi E_0)\le C\sqrt\phi.
\end{equation}
Also set $h'_{ij}=z_{i+1}\cdot z_{j+1}$ and
$b'_{ij}=(-h'_{ij})_+$ for $i,j\in I$.
Expanding the pair constraint with zero envelope gives
\[
 \frac{\widetilde r_i\cdot \widetilde r_j}{a_i a_j}
 \ge h_{ij}-\frac{\delta_i+\delta_j}{2a_i a_j}.
\]
At starts in $I$, \eqref{eq:good-move-geometry} gives
$\widetilde r_i/a_i=z_{i+1}+O(\sqrt\phi)$.
Since $\max\{a_i/a_j,a_j/a_i\}\le6E_0$, we obtain
\begin{equation}\label{eq:sign-transport}
 h'_{ij}\ge h_{ij}-C(\sqrt\phi+\phi E_0)
                  \ge h_{ij}-C\sqrt\phi,
 \qquad
 b_{ij}b'_{ij}\ge(b'_{ij})^2-C\sqrt\phi\,b'_{ij}.
\end{equation}
Set $P=200\sqrt\phi(d+10^5\phi E_0/\bar\epsilon^2)$.
Equations~\eqref{eq:global-Gram-mass}--\eqref{eq:sign-transport} give
\begin{equation}\label{eq:total-positive-error}
 \begin{gathered}
 \sum_{\substack{i,j\in E\\i<j}}(h_{ij})_+^2\le P,
                          \qquad E=I,\ I+1,\\
 \sum_{\substack{i,j\in I\\i<j}}
       \bigl((b'_{ij})^2-b_{ij}b'_{ij}\bigr)_+\le P,\qquad
 \sum_{\substack{i,j\in I\\i<j}}b_{ij}(h'_{ij})_+\le P.
 \end{gathered}
\end{equation}
For the last assertion, multiply the maximum positive correlation on
$I+1$ by the total negative sum on $I$. Dividing a predictor by
$\max\{1,\|V_j\|\}$ only decreases this positive sum.

For any suffix of the prefix under consideration,
\eqref{eq:bad-start-count} and \eqref{eq:lower-log-growth} now give
\begin{equation}\label{eq:suffix-growth}
 \log(a_T/a_u)
 \le(T-u)\epsilon+\tau(P+N)
                         -(1-3\tau/\phi)S_{u,T}.
\end{equation}
The sum of the positive predictor terms over $j\in I$ is at most $P$.
The negative terms only decrease the sum. Outside $I$, each contribution is
at most one because $\|p_j\|\le1$. The preceding count bounds the number
of such indices on $[u,T)$.
The parameter choices give $1-3\tau/\phi\ge23/32$.
Step~4 verifies the error budgets
$P<d/(64\beta^2)$, $\tau(P+N)<1/8$, and $P+N<d/(256\beta)$.
Applying \eqref{eq:suffix-growth} for every $u<T$ excludes a
first violation of \eqref{eq:radius-bootstrap}. Thus that bound holds
throughout every prefix before contraction.

Using $\log a_T\ge-d\epsilon$, $K\le d(1+2/\beta)$, and
$\tau(P+N)\le d\epsilon/16$ in \eqref{eq:suffix-growth} with $u=0$
gives $S_{0,T}<3d\epsilon$.
Equation~\eqref{eq:bad-start-count} excludes at most $3d/(40\beta)$ starts
because $N+9d\epsilon/\phi<3d/(40\beta)$.
If $T=K$, put $\alpha_j=2/\max\{2,\|V_j\|\}$ for $j\in I$.
The predictor bound gives $\alpha_j\ge1-e_j$, so
$m=\sum_{j\in I}\alpha_j\ge|I|-\sum_{j\in I}e_j\ge d+9d/(10\beta)$.
The generalized Welch inequality \cite{Waldron2003}, applied to
$\sqrt{\alpha_j}z_{j+1}$, gives
\[
 \sum_{\substack{i,j\in I\\i<j}}\alpha_i\alpha_j(h'_{ij})^2
 \ge\frac{m^2/d-\sum_{j\in I}\alpha_j^2}{2}
 \ge\frac{m^2/d-m}{2}\ge9d/(20\beta).
\]
Since $0<\alpha_j\le1$, these weights cannot increase the sums of
positive correlation squares and positive parts of the deficits in
\eqref{eq:total-positive-error}. Subtracting the two error bounds from
the Welch estimate gives
\[
 \sum_{\substack{i,j\in I\\i<j}}\alpha_i\alpha_j b_{ij}b'_{ij}\ge2d/(5\beta).
\]
For $i<j$ in $I$, the normalized predictor coefficient satisfies
$1/\max\{1,\|V_j\|\}\ge\alpha_j/2\ge\alpha_i\alpha_j/2$.
Thus the total negative predictor contribution is at least $d/(5\beta)$.
Bound the positive terms and indices outside $I$ as in
\eqref{eq:suffix-growth}. Then
\[
 \log a_K\le K\epsilon+\tau(N+P-d/(5\beta))
                         -(1-3\tau/\phi)S_{0,K}
 < -d\epsilon\le\log a_* ,
\]
using $\tau\ge11\beta\epsilon$, $K\le d(1+2/\beta)$, and
$1/5-1/256>3/16$. This contradicts the absence of
contraction. Every stage with zero envelope reaches query norm at most $a_*$ within
$K$ moves.

The controller uses bounded controls up to contraction or time $K$.
Indeed, $a_i\le[L(1+\tau)]^K$ and $\|c_i\|\le La_i$ in
\eqref{eq:lower-control}. Complete the rule with zero control after
contraction. Lemma~\ref{lem:finite-contraction}, with $N=K$ and $q=a_*$,
therefore prevents escape at slope $L$. By \eqref{eq:Gamma} and
finite-history invariance, this proves the
lower estimate in \eqref{eq:dimension-remainder}.

\medskip
\noindent\textit{Step 4. Verification of constants.}
For $\beta>85$, the parameter choices give
$K\le d(1+2/\beta)\le2d$, $N\le2\beta$, and
\[
 11\beta\epsilon\le\tau\le6\beta^2/d,\qquad
 128\beta/d\le\phi\le64\beta^2/d,\qquad E_0\le3\exp(-32)\sqrt d/\beta^3 .
\]
The bound on $E_0$ follows from
$K\epsilon\le \beta/2-3\log \beta-31$.
The same parameter bounds give $N/c_*\le1/\phi$,
$\tau^2\le\phi$, and $3\tau/\phi\le9/32$.

In \eqref{eq:good-move-geometry}, the deficit bounds hold with
$8s_i a_i^2$, the center bound with $4La_i\sqrt{s_i}$, and the response
error bound with $2\sqrt{s_i}+\tau$. For the recurrence for $F_l$,
the two deficit terms contribute at most $136g_{l+1}U_lW_l$.
The squared norm of the weighted response sum contributes at most
$2\lambda(1+\tau)^2(A+F_{l+1})+
\lambda(16g_{l+1}+4\tau^2)U_lW_l$.
Using $\tau^2\le\phi$ and $U_lW_l\le\exp(l/8)U_0W_0$ in the recurrence
for $F_l$ bounds its additive term by
$276\phi(2\exp(1/8))^lU_0W_0$ and its coefficient on $A+F_{l+1}$ by
$3/16$. Since $(3/8)\exp(1/8)<1/2$ and
$K(3/16)^n<1/12$, \eqref{eq:weighted-Gram} holds with $C_0=1000$ and
$\theta=1/3$.

Thus both sums in \eqref{eq:global-Gram-mass} are at most
$2d+200000\phi E_0/\bar\epsilon^2$ for unordered pairs.
Since $4/3+1000\phi|I|\le2^{11}\phi d$,
\eqref{eq:row-error-moments} gives
\[
 \sum_{j\in I}e_j\le10^6\phi^{3/2}dE_0/\bar\epsilon<d/(64\beta),
\]
using $\sqrt{1000\cdot200}\,2^{11}<10^6$ and
$\phi^{3/2}E_0/\bar\epsilon\le3072\exp(-32)/\beta$.
For $i\in I$, the ratios $a_{i+1}/a_i$ vary by a factor at most
$\exp(\phi+\epsilon+3\tau)<2$, giving the shifted error bound $2e_j^2$.

Since $\sqrt\phi E_0<1/100$, the singleton bound in
\eqref{eq:pointwise-positive} is at most
$(17\sqrt\phi+96\phi E_0)/2\le100\sqrt\phi$. For
\eqref{eq:sign-transport}, the response error in replacing $\widetilde r_i/a_i$ by
$z_{i+1}$ is at most $5\sqrt\phi$, and the deficit term is bounded by
$48\phi E_0$. Thus all three estimates in
\eqref{eq:total-positive-error} hold with the stated $P$.
Indeed, $\sqrt\phi\le8\beta/\sqrt d$, and $\log \beta\le \beta/18$ for $\beta>85$ gives
\[
 d\bar\epsilon=\max\{1,d\epsilon\}
 \ge\frac{32+d\epsilon}{33}
 =\frac{\beta/2-3\log \beta}{33}\ge\frac \beta{99}.
\]
Using $\phi=128\beta\bar\epsilon$ and
$E_0\le3\exp(-32)\sqrt d/\beta^3$, we bound $P\beta^2/d$ by
$1600\beta^3\exp(-\beta/2)+10^{12}\exp(-32)<1/64$.
Together with $N\le2\beta$,
this gives $P+N<d/(256\beta)$ and $\tau(P+N)<1/8$, as used in the
contraction argument.

The upper and lower estimates prove \eqref{eq:dimension-remainder}.
The expansion and limit in \eqref{eq:dimension-leading} follow by dividing
by $d$ and by $\log d$, respectively.
\end{proof}

\bibliographystyle{siamplain}
\bibliography{references}

\end{document}